\documentclass[10pt,a4paper]{article}
\usepackage[utf8]{inputenc}
\usepackage{amsmath}
\usepackage{amsfonts}
\usepackage{amssymb}
\usepackage{amsthm}
\usepackage{enumitem}
\usepackage{color}
\usepackage{hyperref}
\usepackage{MnSymbol}
\usepackage{graphicx}

\usepackage[a4paper, textwidth=13cm]{geometry}

\newcommand{\N}{\mathbb{N}}
\newcommand{\R}{\mathbb{R}}
\newcommand{\Z}{\mathbb{Z}}
\renewcommand{\oe}{\Omega_\epsilon}
\newcommand{\osem}{\Omega_{\ast,\epsilon}^M}
\newcommand{\oem}{\Omega_{\epsilon}^M}
\newcommand{\oeh}{\Omega_{\epsilon,h}}
\newcommand{\ssepm}{S_{\ast,\epsilon}^\pm}
\newcommand{\ueps}{u_{\epsilon}}
\newcommand{\uepsm}{u_{\epsilon}^M}
\newcommand{\fxe}{\frac{x}{\epsilon}}
\newcommand{\neps}{N_{\epsilon}}
\newcommand{\x}{\bar{x}}
\newcommand{\Leps}{\mathcal{L}_{\epsilon}}
\newcommand{\Heps}{\mathcal{H}_{\epsilon}}
\newcommand{\veps}{v_{\epsilon}}
\newcommand{\weps}{w_{\epsilon}}
\newcommand{\foe}{\frac{1}{\epsilon}}
\newcommand{\peps}{\phi_{\epsilon}}
\newcommand{\ie}{i.\,e.,\,}
\newcommand{\Hepsom}{\mathcal{H}_{\epsilon,0}^M}
\newcommand{\who}{\widehat{\Omega}}
\newcommand{\whs}{\widehat{\Sigma}}
\newcommand{\Hepsho}{\mathcal{H}_{\epsilon,h,0}}
\newcommand{\Lepsh}{\mathcal{L}_{\epsilon,h}}
\newcommand{\pepsh}{\phi_{\epsilon,h}}
\newcommand{\per}{\mathrm{per}}
\newcommand{\te}{\mathcal{T}_{\epsilon}}
\newcommand{\ue}{\mathcal{U}_{\epsilon}}
\newcommand{\Ho}{\mathcal{H}_0}
\newcommand{\bxi}{\bar{\xi}}

\newtheorem{definition}{Definition}
\newtheorem{remark}{Remark}
\newtheorem{theorem}{Theorem}
\newtheorem{proposition}{Proposition}
\newtheorem{lemma}{Lemma}
\newtheorem{corollary}{Corollary}

\title{Singular limit for reactive diffusive transport through an array of thin channels in case of critical diffusivity}

\author{M. Gahn and M. Neuss-Radu}
\date{}

\begin{document}

\maketitle

\begin{abstract}
We consider a nonlinear reaction--diffusion equation in a domain consisting of two bulk regions connected via small channels periodically distributed within a thin layer. The height and the thickness of the channels are of order $\epsilon$, and the equation inside the layer depends on the parameter $\epsilon$. We consider the critical scaling of the diffusion coefficients in the channels and nonlinear Neumann-boundary condition on the channels' lateral boundaries. We derive effective models in the limit $\epsilon \to 0 $, when the channel-domain is replaced by an interface $\Sigma$ between the two bulk-domains. Due to the critical size of the diffusion coefficients, we obtain jumps for the solution and its normal fluxes across $\Sigma$, involving the solutions of local cell problems on the reference channel in every point of the interface $\Sigma$.
\end{abstract}

\\


\section{Introduction}

In this paper we consider reaction-diffusion processes in a microscopic domain $\oe$ consisting of two bulk-domains $\oe^+$ and $\oe^-$, which are connected via small periodically distributed channels $\osem$ obtained by scaled and shifted reference elements $Z^{\ast}$.  The height and the thickness of the channels, as well as the periodicity of the channels, are of order $\epsilon$, where the parameter $\epsilon$ is small compared to the size of the bulk-domains.  Within the microscopic domain $\oe$ 
we consider a reaction-diffusion equation with nonlinear reaction-kinetics and nonlinear Neumann-boundary condition on the lateral boundary of the channels. This boundary condition describes for example reactions taking place at the boundary of the channels or exchange with the surrounding medium. Within the channels we assume low diffusivity of order $\epsilon$. The aim is the derivation of a macroscopic model with effective interface conditions in the singular limit $\epsilon \to 0$. We only consider the case of a scalar equation. However, the results can be easily extended to systems of equations.

Reaction-diffusion transport in domains connected by thin channels plays an important role in many applications. We are particularly interested in applications in biology and biomedicine, where, for example, the exchange between cells and extracellular space occurs through pores in the cell membrane, or where cell layers such as the blood-brain barrier or the blood-air barrier control the exchange of chemical substances, ions, fluids, or cells (e.g., immune cells) between compartments of the organism. Quantifying the barrier function of such layers in dependence also on microenvironmental influences (pH, hypoxia) leads to extremely challenging problems. First experimental investigations in this direction are performed with the help of \textit{in vitro} devices consisting of  microchambers separated by cell layers cultured on porous membranes  see e.g. Figure \ref{ibidi}. These so called organs-on-chips simulate structural and functional features of \textit{in vivo} cellular layers, see e.g. \cite{Huh2010} for a lung-on-chip microdevice reproducing key structural and functional properties of the human alveolar-capillary interface. Other important applications where reaction-diffusion transport through small channels play an important role are membranes perforated by tiny pores and filters used in engineering sciences, see e.g. Figure \ref{FigurNanoMembrane}. In all these applications nonlinear effects on the boundary of the channels play an important role.  We note that in the mentioned applications the geometry and the bio-chemical and bio-physical processes can be considerably more complicated than in our model. However, even for relatively simple models, just taking into account reaction and diffusion processes, numerical simulations are very expensive.  Therefore, macroscopic approximations of the solutions, obtained in the limit $\epsilon \to 0$, are highly demanded and the methods derived in our paper  are an essential step  for the treatment of more complex problems.
\begin{figure}
\begin{center}
\includegraphics[scale=0.3]{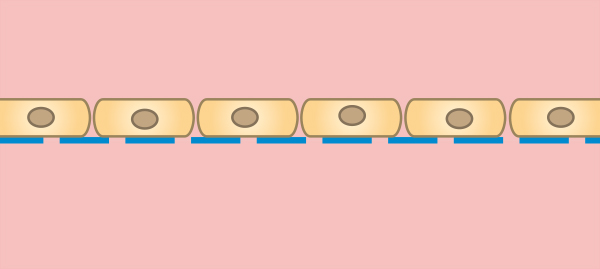} \quad 
\includegraphics[scale=0.2]{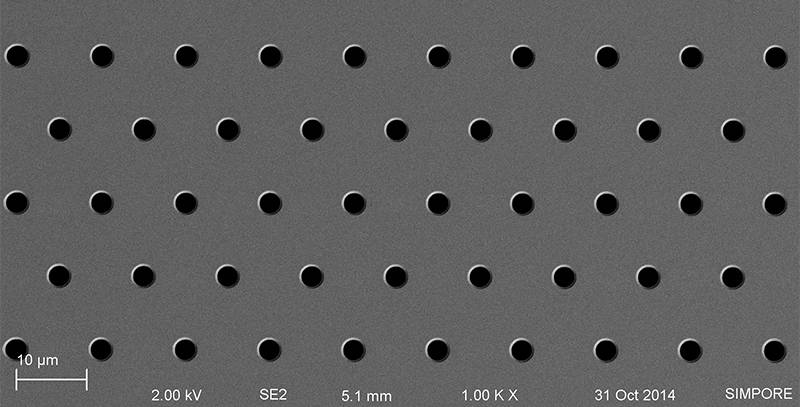} 
\caption{Left: Endothelial barrier assay: Human endothelial cell monolayer cultivated on ibiPore  glas membrane separating the upper and lower microchamber. Right: Geometry of the supporting porous membrane. Bar: 10 $\mu$m, diameter of the single pore: 3$\mu$m. With the kind approval of ibidi GmbH.}
\label{ibidi}
\end{center}
\end{figure}
\begin{figure}[h]
\begin{center}
\includegraphics[scale=0.4]{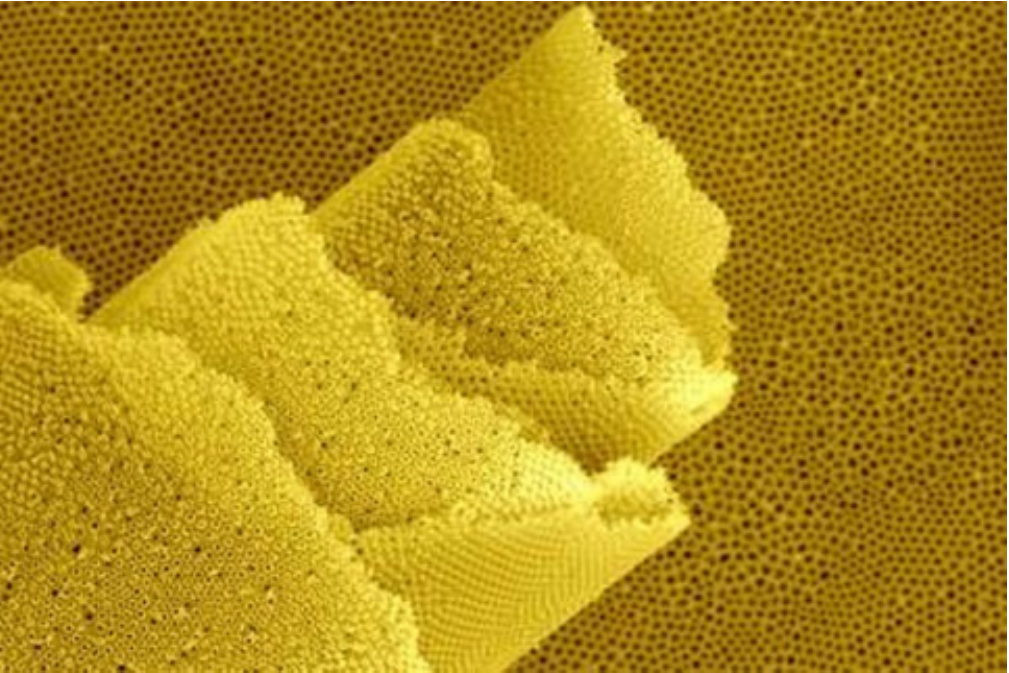}  
\caption{Two rolled up, nested gold (Au)-membranes with ordered pores. \copyright Claudia B\"uttner, Max-Planck-Institut f\"ur Mikrostrukturphysik, Halle (Saale). With the kind approval of science2public - Gesellschaft f\"ur Wissenschaftskommunikation.}
\end{center}
\label{FigurNanoMembrane}
\end{figure}

In the singular limit $\epsilon \to 0$ we obtain two bulk-domains $\Omega^+$ and $\Omega^-$, which are separated by the interface $\Sigma$. The crucial point is the derivation of the interface conditions across $\Sigma$ in the macroscopic model. These interface conditions carry information about the microscopic processes in the channels $\osem$. 
We obtain a jump for the solution and its normal fluxes across $\Sigma$, involving the solution of local cell problems of reaction-diffusion type on the reference element $Z^{\ast}$ in every point of the interface $\Sigma$. This coupling between the microscopic variable from $Z^{\ast}$ and the macroscopic variable from $\Sigma$ arises due to critical scaling for the diffusion in the microscopic model, and leads to additional difficulties in the limit process. Here, the most challenging step is to pass to the limit in the thin channels, where we have to cope simultaneously with the singular limit and the periodicity of the channels. For this we use the method of two-scale convergence for thin channels and their oscillating surface, which was defined in \cite{BhattacharyaGahnNeussRadu} and is closely related to the two-scale convergence in thin heterogeneous layers introduced in \cite{NeussJaeger_EffectiveTransmission}, see also \cite{MarusicMarusicPalokaTwoScaleConvergenceThinDomains} for homogeneous thin structures. Due to the specific scaling in the microscopic equation in the channels, the macroscopic two-scale limit of the microscopic solution is depending on both, the macroscopic variable $\x \in \Sigma$ and the microscopic variable $y \in Z^{\ast}$.

As a first step in the derivation of the macroscopic model we derive $\epsilon$-dependent a priori estimates for the microscopic solutions, which imply (weak) two-scale convergence for the solutions in the thin layer. These compactness results are enough to pass to the limit in the linear terms of the microscopic equation, but not for the nonlinear terms, especially on the boundary of the channels. For those terms we need strong two-scale compactness results. Using the unfolding method for thin channels, which gives us an equivalent characterization of the two-scale convergence, we prove a general strong two-scale compactness result of Kolmogorov-Simon-type based on error estimates for the discrete shifts of the microscopic solution. An additional difficulty in our problem arises due to the fact, that because of the low regularity assumptions on the data and the nonlinear boundary condition the time-derivative of the microscopic solution is only a functional  on a function space defined on the whole domain $\oe$. Therefore, it is not straightforward to obtain the existence of the time-derivative of the unfolded sequence together with suitable a priori estimates. To overcome this problem we use a duality argument to show that the time-derivative of the unfolded sequence  in the channels exists and can be controlled by the time-derivative of the microscopic solution on the whole domain.  To exhibit the form of the macroscopic model in the limit $\epsilon \to 0$, we construct test-functions adapted to the structure of the transmission problem and which are admissible for the definition of the two-scale convergence in the channel domain. These test-functions lie in a function space which we show to be dense in the space of macroscopic solutions.

First homogenization results for problems with a geometrical framework related to our setting were given in \cite{sanchez70boundary}. In \cite{ConcaI1987,ConcaII1987}, the Stokes-equation was considered in two-bulk domains separated by a sieve of thickness zero. Contributions to the homogenization of the Laplace equation in domains connected by thin channels have been given in \cite{del1987thick, Onofrei2006, Shaposhnikova2001, yablokov2004problem}, where the asymptotic behavior of the solution is investigated for different ratio of the thickness of the layer and the radius of the cylindrical channels. The more challenging problem concerning the ion transport through channels of biological membranes was announced in \cite{neuss2005homogenization}. The homogenization of an elliptic Steklov type spectral problem in domains connected by thin channels was considered in \cite{amirat2016asymptotics, gadyl2018asymptotic}. 
In \cite{MatzavinosPtashnyik2016} a homogenization problem for  diffusion-advection processes for oxygen transport through skin layer (heterogeneous thin layer) and fat tissue (heterogeneous bulk domain) is considered. The problem is linear and the case of high diffusivity of order $\epsilon^{-1}$ was investigated. The case of transport through channels for moderate and high diffusivity, again for linear problems with Neumann boundary conditions, can be found in \cite{BhattacharyaGahnNeussRadu}. More precisley, the diffusion in the channels is of order $\epsilon^{\gamma}$ with $\gamma \in [-1,1)$. In this case, the two-scale limit of the microscopic solution in the channels is not depending on the microscopic variable $y \in Z^{\ast}$. In the present paper we treat the critical case $\gamma = 1$ of small diffusion (leading to a coupled micro-macro model), and additionally take into account nonlinear reaction-kinetics in the bulk-domains and in the channels, as well as on the boundary of the channels.

Reaction-diffusion problems through a thin layer instead of channels were considered in  \cite{GahnEffectiveTransmissionContinuous,GahnNeussRaduKnabner2018a,NeussJaeger_EffectiveTransmission}.  
In \cite{NeussJaeger_EffectiveTransmission} continuous transmission conditions between the bulk and the layer and nonlinear reaction kinetics are considered.  The results are based on uniform $L^{\infty}$-estimates for the solution and $L^2$-estimates for the time-derivative. In \cite{GahnNeussRaduKnabner2018a} a nonlinear interface condition between the thin layer and the bulk domains is considered, leading to $(H^1)'$-regularity for the time-derivative. 
In our paper we extend those results to channels with continuous transmission conditions to the bulk domains, nonlinear boundary conditions, especially on the channels lateral boundaries, and low regularity for the time-derivative (operators on function spaces defined on the whole microscopic domain $\oe$), in the critical case of low diffusion in the channels, see also \cite{GahnDiplomarbeit}. For this purpose we derive general strong two-scale compactness results of Kolmogorov-Simon-type which allow us to avoid the use of $L^{\infty}$-estimates.

The structure of the paper is as follows. In Section \ref{SectionMicroscopicModel} we introduce the microscopic model with its underlying geometrical structure, and give the definition of a weak solution and the necessary function spaces. Further, we prove existence and uniqueness of a unique solution and establish a priori estimates for these solutions and their shifts. In Section \ref{Sect_Two_scale_convergence} we give the definition of the two-scale convergence and the unfolding operator for thin channels. Further, we derive weak and strong two-scale compactness results associated with the specific scaling in our microscopic model. Finally in Section \ref{SectionDerivationMacroscopicModel} we derive the macroscopic model.

\subsection{Highlights and original contributions}
In this paper we address the dimension reduction and homogenization of a reaction-diffusion model in a domain consisting of two bulk regions connected by an array of thin channels with critical scaling of order $\epsilon$ of the diffusivity in the channels. This scaling leads in the limit $\epsilon \to 0$ to a macroscopic model with effective interface transmission conditions which couple the micro and macro variables. We derive general strong two-scale compactness results of Kolmogorov-Simon-type, which are needed also for more complex nonlinear problems involving transmission processes through channels, arising e.g. from bio-medical or engineering applications. The main results of the paper are:
\\
\begin{enumerate}
[label = - ]
\item Existence of a microscopic solution with uniform a priori estimates with respect to $\epsilon$ and estimates for  the differences between the shifted solution and the solution itself, see Section \ref{SectionExistenceAprioriEstimates};
\item Commuting property between the unfolding operator and the generalized time-derivative based on a general duality argument, see Lemma \ref{LemmaGeneralTimeDerivative} and Proposition \ref{PropositionExistenceTimeDerivativeUnfoldedSequence};
\item A general strong multi-scale compactness result of Kolmogorov-Simon-type under low regularity assumptions on the time-derivative, see Theorem \ref{TheoremStrongTSConvergence}; 
\item Derivation of a macroscopic model including effective transmission conditions across $\Sigma$, where the jump of the solution and its flux across $\Sigma$ is coupled to local cell problems of reaction-diffusion type, see Theorem \ref{MainResultMacroscopicModel};
\item Density result for a class of test functions adapted to the micro-macro structure of the macroscopic model, see Proposition \ref{Proposition_Density}.
\end{enumerate}

\section{The microscopic model}
\label{SectionMicroscopicModel}
	
Let $\epsilon > 0$ be a sequence of positive numbers tending to zero such that $\epsilon^{-1}\in \N$ and let $H>0$ be a fixed real number. Let $n\in \N, n\geq2$. For $x\in \R^n$, we write $x=(\x, x_n) \in \R^{n-1}\times \R$.
Let $\Omega$ be a subset of $\R^n$ defined as
\begin{align*}
\Omega := \Sigma \times (-H,H),
\end{align*}
where $\Sigma\subset \R^{n-1}$ is a connected and open domain with Lipschitz-boundary.

	
We consider the domain $\oe \subset \Omega$ consisting of three subdomains: the bulk regions $\oe^+$ and $\oe^-$ which are connected by channels periodically distributed within a thin layer constituting the domain $\osem$, see Figure \ref{Figure_Microscopic_Domain}. 
\begin{figure}[h]
\centering
\begin{minipage}[t]{0.43\textwidth}
\includegraphics[width=\textwidth]{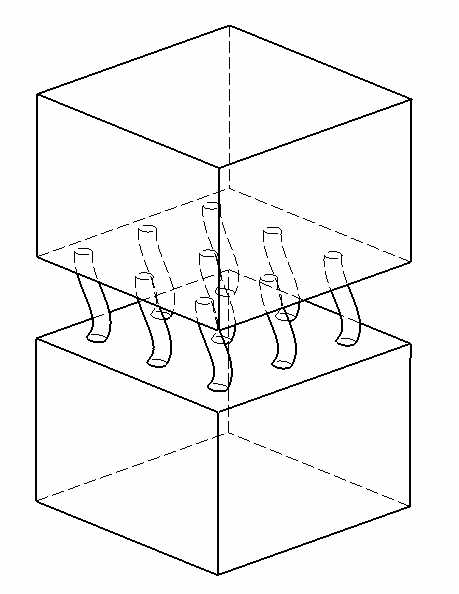}
\caption{Microscopic domain $\oe$ for the case $\epsilon= \frac{1}{3}$ and $n=3$.}\label{Figure_Microscopic_Domain}
\end{minipage}
\begin{minipage}[t]{0.05\textwidth}
\mbox{}
\end{minipage}
\begin{minipage}[t]{0.43\textwidth}
\includegraphics[width=\textwidth]{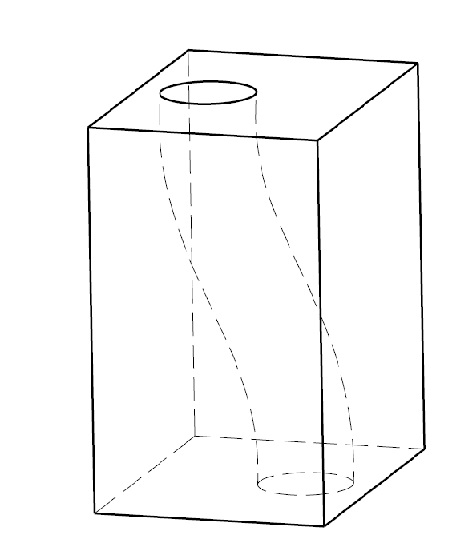}
\caption{Standard channel domain $Z^{\ast}$ in the standard cell $Z$.}\label{Figure_Channel}
\end{minipage}
%
%
\end{figure}
The bulk regions are given by: 
\begin{align*}
	\oe^+  = \Sigma \times (\epsilon,H), \quad	\oe^-  =  \Sigma \times (-H,-\epsilon).
\end{align*}
Furthermore, we denote by 
\begin{align*}
	S_\epsilon^+ = \Sigma \times \{\epsilon\}, \quad 	S_\epsilon^- = \Sigma\times \{-\epsilon\},
\end{align*}
the bottom and the top of $\oe^+$ and $\oe^-$, respectively.
The thin layer separating the two bulk domains is given by
\begin{align*}
\oem := \Sigma \times (-\epsilon, \epsilon).
\end{align*}
To define the channels, which are periodically distributed within $\oem$, we first define the standard cell
\begin{align*}
Z:=Y \times  (-1,1) := (0,1)^{n-1} \times  (-1,1)
\end{align*}
with the upper and lower boundaries
\begin{align*}
S^{\pm}:= Y \times \{\pm 1\}.
\end{align*}
Let $Z^{\ast} \subset Z$ be a connected and  open Lipschitz domain representing the standard channel domain, see Figure \ref{Figure_Channel}, such that
\begin{align*}
S_{\ast}^\pm :=\{y \in \partial Z^{\ast} : y_n=\pm 1\} 
\end{align*}
is a Lipschitz-domain in $\R^{n-1}$ with positive measure.
Let the lateral boundary of the standard channel be denoted by 
\begin{align*}
N:= \partial  Z^{\ast} \setminus (S_{\ast}^+ \cup S_{\ast}^-).
\end{align*}
We assume that this lateral boundary has a non-zero distance to the lateral boundary  $\partial Z \setminus (S^+ \cup S^-))$ of the standard cell $Z$. 
The domain consisting of the channels is then given by
\begin{align*} 
\osem : = \bigcup\limits_{\bar{k} \in  I_\epsilon} \epsilon \big(Z^{\ast}+(\bar{k},0)\big) 
\end{align*}  
where $I_\epsilon= \{ \bar{k}\in \Z^{n-1}: \epsilon \big(Z+(\bar{k},0)\big) \subset \oe^M \}$. The interfaces between the channel domain  and the bulk domains are defined by 
\begin{align*}
\ssepm := \bigcup\limits_{\bar{k} \in  I_\epsilon} \epsilon \big(S_{\ast}^\pm + (\bar{k},0)\big).
\end{align*}  
The microscopic domain $\oe$ is thus defined by
\begin{align*}
\oe= \oe^+ \cup \oe^-\cup \osem \cup S_{{\ast}, \epsilon}^+ \cup S_{{\ast}, \epsilon}^-.
\end{align*}
We assume $\oe$ to be Lipschitz. The boundary of $\oe$ is given by $\partial \Omega_ \epsilon= N_\epsilon\cup\partial_{N} \oe$, where 
\begin{align*}
N_\epsilon  := \bigcup\limits_{k \in  I_\epsilon} \epsilon \big(N + (k,0)\big), \quad \partial_{N} \oe  := \partial \Omega_ \epsilon \setminus N _\epsilon.
\end{align*}
For a function defined on $\oe$, we usually add superscripts $+,-,M$ to denote its restriction to the sub-domains $\oe^+, \oe^-$ and $\osem$ respectively. 
	Finally, we define the domains
\begin{align*}
\Omega^+ := \Sigma \times (0,H), \quad \Omega^- := \Sigma \times (-H,0).
\end{align*}
which are separated by the interface $\Sigma$.

Let $p \in [1,\infty]$ and $p'$ denotes the dual exponent of $p$, \ie $\frac{1}{p} + \frac{1}{p'} = 1$.
For a suitable subset $G\subset \R^m$ with $m\in \N$ we use the following notation for $u \in L^p(G) $ and $v \in L^{p'}(G)$:
\begin{align*}
(u,v)_G:= \int_G u (z) v(z) dz,
\end{align*}
and for $p=p'=2$ we just obtain the inner product on $L^2(G)$. For a Banach space $X$ we denote its dual by $X'$ and by $\langle \cdot , \cdot \rangle_{X',X}$ the duality pairing between $X'$ and $X$, \ie for $x'\in X'$ and $x \in X$ we write
\begin{align*}
\langle x' , x \rangle_{X',X} := x'(x).
\end{align*}
With the subscript $\per$ we indicate function spaces of functions defined on subsets of $\R^n$ which are periodic with respect to the first $(n-1)$ components. For example we write $C_{\per}(\overline{Z^{\ast}})$ for the space of functions $C(\overline{Z^{\ast}})$ periodically extended in $\bar{y}$-direction. 

For two values $a^+ \in \R$ and $a^-\in \R$ we use the notation
\begin{align*}
\sum_{\pm} a^{\pm} := a^+ + a^-.
\end{align*}

\subsection{The microscopic model}

We study the following reaction-diffusion problem  for the unknown function $\ueps= (\ueps^+,\uepsm,\ueps^-):(0,T)\times \oe \rightarrow \R$:
\begin{subequations}\label{MicroscopicModel}
\begin{align}\label{micro_reac_diff_eq}
\begin{aligned}
\partial_{t} \ueps^{\pm} -D^{\pm} \Delta \ueps^{\pm} &=f^{\pm}(\ueps^{\pm}) &\mbox{ in }& (0,T)\times \oe^{\pm},
\\
\frac{1}{\epsilon}\partial_{t}\uepsm - \nabla\cdot\left(\epsilon  D^M\left(\frac{x}{\epsilon}\right) \nabla \uepsm\right) &=  \frac{1}{\epsilon}g_\epsilon(\uepsm) &\mbox{ in }& (0,T)\times \osem,
\end{aligned}
\end{align}
with the boundary conditions
\begin{align}\label{micro_boundary_cond}
\begin{aligned} 
\nabla \ueps^\pm  \cdot\nu &=0 &\mbox{ on }& (0,T) \times \partial_{N}\oe ,
\\
-\epsilon D^M\left(\frac{x}{\epsilon}\right)\nabla \ueps^M \cdot \nu &= h_\epsilon(\uepsm) &\mbox{ on }& (0,T) \times N_\epsilon,
\end{aligned}
\end{align}
where $\nu$ denotes the outer unit normal with respect to $\oe$, and the initial conditions
\begin{align} \label{micro_initial_cond_pb_1}
\ueps(0) =  u_{\epsilon,i} \quad \mbox{on } \oe.
\end{align}
At the interfaces $\ssepm$, we impose the natural transmission conditions, i.e., the continuity of the solution and of the normal flux, namely
\begin{align}\label{micro_natural_trans_cond}
\begin{aligned} 
\ueps^\pm &=\ueps ^M &\mbox{ on }& (0,T) \times \ssepm
\\
D^\pm \nabla \ueps^\pm\cdot\nu &= \epsilon D^M \left(\frac{x}{\epsilon}\right)\nabla \ueps^M\cdot\nu &\mbox{ on }&(0,T) \times \ssepm,
\end{aligned}
\end{align}
\end{subequations}	
where here $\nu$ denotes the unit normal on $\ssepm$ (with respect to $\osem$).
We emphasize that we  consider constant diffusion coefficients $D^{\pm}$ in the bulk-domains $\oe^{\pm}$ just for an easier notation. The results can be easily extended to more general problems, for example oscillating diffusion coefficients.

\subsubsection{Assumptions on the Data}
\begin{enumerate}[label = (A\arabic*)]
\item\label{AssumptionDiffusion} 	We assume   $D^\pm >0$ and $D^M:  Z^{\ast} \to \mathbb{R}^{n\times n}$ with $D^M \in L^{\infty}(Z^{\ast})^{n\times n}$ $Y$-periodic and  coercive, that means for almost every $y \in Z^{\ast}$ and all $\xi \in \R^n$ it holds for $c_0>0$ independent of $y$ and $\xi$ that
\begin{align*}
D^M(y)\xi \cdot \xi \geq c_0 \|\xi\|^2.
\end{align*}
%
\item\label{AssumptionReactionRateBulk}  For the reaction rates in the bulk domains we assume $ f \in C^0([0,T]\times \overline{\Omega} \times \R)$ is globally Lipschitz-continuous with respect to the last variable.
\item\label{AssumptionReactionRateChannels}  For the reaction rate in the channels domain we assume $g_{\epsilon}(t,x,z)= g\left(t,\fxe,z\right)$ for $(t,x,z)\in (0,T) \times \osem \times \R$, with $g \in C^0([0,T] \times \overline{Z^{\ast}} \times \R)$ $Y$-periodically extended in the second variable. Further, we assume that $g$ is globally  Lipschitz-continuous with respect to the last variable.
\item\label{AssumptionReactionRateSurfaceChannels}  For the reaction rate on the lateral surface of the channels we assume that $h_\epsilon (t,x,z):= h\left(t,\fxe,z \right)$ for $(t,x,z) \in (0,T)\times \neps \times \R$, with $h \in C^0([0,T] \times \overline{N} \times \R)$ extended $Y$-periodically with respect to the second variable, and globally Lipschitz-continuous with respect to the last variable.
\item\label{AssumptionInitialCondition}  For the initial conditions, we assume that
\begin{align*}
u_{\epsilon,i}(x)= 
\begin{cases}
u_i^+(x) &\mbox{ for } x \in \oe^+,
\\
u_i^M\left(\x,\frac{x}{\epsilon}\right) &\mbox{ for } x \in \osem,
\\
u_i^-(x) &\mbox{ for } x \in \oe^-,
\end{cases}
\end{align*} 
where $(u_i^+,u_i^M, u_i^-) \in L^2(\Omega^+)\times L^2 (\Sigma ,C^0(\overline{Z^{\ast}}))\times L^2(\Omega^-) $. Especially, it holds that
\begin{align*}
\frac{1}{\sqrt{\epsilon}} \left\|u_i^M\left(\cdot, \frac{\cdot_{x}}{\epsilon}\right)\right\|_{L^2(\osem)} \le \|u_i^M\|_{L^2(\Sigma,C^0(\overline{Z^{\ast}}))} \le C.
\end{align*}
\end{enumerate}

\subsection{Weak solution of the microscopic problem}

In this section we define a weak solution of the microscopic model $\eqref{MicroscopicModel}$. Therefore, we introduce Hilbert spaces with inner products adapted to the scaling in the equation for $\uepsm$ in the channels. First of all, we define 
\begin{align*}
\Leps := L^2(\oe) = L^2(\oe^+)\times L^2(\osem)\times L^2(\oe^-),
\end{align*}
together with the inner product 
\begin{align*}
(\weps,\veps)_{\Leps}:= (\weps,\veps)_{\oe^+} + (\weps,\veps)_{\oe^-} + \foe (\weps,\veps)_{\osem}.
\end{align*}
The second equality in the definition of $\Leps$ means that we always identify $L^2(\oe)$ with the product space on the different compartments. Further, we define
\begin{align*}
\Heps:=   \left\{ (\ueps^+,\uepsm,\ueps^-) \in H^1(\oe^+)\times H^1(\osem) \times H^1(\oe^-) \, : \, \ueps^{\pm} = \uepsm \mbox{ on } \ssepm\right\},
\end{align*}
with the inner product
\begin{align*}
(\weps,\veps)_{\Heps} := (\weps,\veps)_{\Leps} + (\nabla \weps ,\nabla \veps)_{\oe^+ } + (\nabla \weps,\nabla \veps)_{\oe^-} + \epsilon (\nabla \weps,\nabla \veps)_{\osem}.
\end{align*}
Of course, it holds $\Heps = H^1(\oe)$ in a topological sense.
The associated norms on $\Leps$ and $\Heps$ are denoted by $\|\cdot \|_{\Leps}$ and $\|\cdot\|_{\Heps}$. We immediately obtain the Gelfand-triple
\begin{align}\label{GelfandTriple}
\Heps \hookrightarrow \Leps \hookrightarrow \Heps'.
\end{align}
Hence, for $\veps \in L^2((0,T),\Heps)\cap H^1((0,T),\Heps')$ the time-derivative $\partial_t \veps $ is characterized by the identity (see  \cite[Prop. 23.20]{ZeidlerIIA})
\begin{align}
\label{CharacterizationTimeDerivative}
\int_0^T \langle \partial_t \veps , \peps \rangle_{\Heps',\Heps} \psi dt = - \int_0^T (\veps ,\peps)_{\Leps} \psi' dt 
\end{align}
for all $\peps \in \Heps$ and $\psi \in \mathcal{D}(0,T)$.

\begin{definition}\label{DefinitionWeakSolution}
We say that $\ueps = (\ueps^+,\uepsm,\ueps^-)$ is a weak solution of the microscopic model $\eqref{MicroscopicModel}$, if 
\begin{align*}
\ueps \in L^2((0,T),\Heps)\cap H^1((0,T),\Heps'),
\end{align*}
for all $\peps \in \Heps$ it holds the variational equation
\begin{align}\label{VariationalEquation}
\begin{aligned}
\langle \partial_t \ueps,\peps\rangle_{\Heps',\Heps}& + \sum_{\pm} \big(D^{\pm} \nabla \ueps^{\pm},\nabla \peps\big)_{\oe^{\pm}} + \epsilon \left(D^M\left(\frac{\cdot}{\epsilon}\right)\nabla \uepsm,\nabla \peps\right)_{\osem}
\\
&= \sum_{\pm} \big(f^{\pm} (\ueps^{\pm}),\peps\big)_{\oe^{\pm}} + \foe \big(g_{\epsilon}(\uepsm),\peps\big)_{\osem} - \big(h_{\epsilon}(\uepsm),\peps\big)_{\neps},
\end{aligned}
\end{align}
and $\ueps$ fulfills the initial condition $\eqref{micro_initial_cond_pb_1}$. This condition makes sense since $\ueps \in C^0([0,T],\Leps)$, see \cite[Prop. 23.23]{ZeidlerIIA} and $u_{\epsilon,i} \in \Leps$ by  Assumption \ref{AssumptionInitialCondition}.
\end{definition}

Let us specify how the weak formulation in Definition \ref{DefinitionWeakSolution} is related to the scaling for the time-derivative in the classical formulation in $\eqref{micro_reac_diff_eq}$. Under the additional regularity assumptions (which is not necessary for our analysis) 
$\partial_t \ueps^{\pm} \in L^2((0,T),H^1(\oe^{\pm})')$ and $\partial_t \uepsm \in L^2((0,T),H^1(\osem)')$ for a weak solution $\ueps$, we obtain from the Gelfand-tripel $\eqref{GelfandTriple}$ and the definition of the inner product on $\Leps$ (see also $\eqref{CharacterizationTimeDerivative}$)
\begin{align*}
\langle \partial_t \ueps,\peps\rangle_{\Heps',\Heps} = \sum_{\pm} \langle \partial_t \ueps^{\pm},\peps\rangle_{H^1(\oe^{\pm})',H^1(\oe^{\pm})} + \foe \langle \partial_t \uepsm,\peps\rangle_{H^1(\osem)',H^1(\osem)}.
\end{align*} 
However, in our case the time-derivative $\partial_t \ueps $ is a functional defined on the whole space $\Heps$, \ie a space of functions defined on the whole domain $\oe$, and it is not straightforward to restrict such functionals to $H^1(\oe^{\pm})'$ and $H^1(\osem)'$. For the derivation of the strong compactness results in the channels we have to control the time-derivative in the domain $\osem$. Therefore, we introduce the space of functions with zero traces on the interface between the channels and the bulk-domains
\begin{align*}
\Hepsom:= \left\{ \veps^M \in H^1(\osem)\, : \, \uepsm|_{\ssepm}=0  \right\},
\end{align*}
together with the inner product
\begin{align*}
(\weps^M,\veps^M)_{\Hepsom} := \foe (\weps^M,\veps^M)_{\osem} + \epsilon (\nabla \weps^M,\nabla \veps^M)_{\osem},
\end{align*}
and the associated norm denoted by $\|\cdot \|_{\Hepsom}$. This leads to the Gelfand-triple $\Hepsom \hookrightarrow L^2(\osem)\hookrightarrow (\Hepsom)'$. By extending functions $\peps^M \in \Hepsom$ by zero to $\oe$, we obtain for $\ueps \in L^2((0,T),\Heps)\cap H^1((0,T),\Heps')$
 (see $\eqref{CharacterizationTimeDerivative}$)

\begin{align*}
\int_0^T\langle \partial_t \uepsm , \peps^M \rangle_{\Heps',\Heps}  \psi(t) dt &=- \int_0^T (  \ueps , \peps^M )_{\Leps} \psi'(t)dt
\\
&= - \frac{1}{\epsilon} \int_0^T (\uepsm,\peps^M)_{L^2(\osem)} \psi'(t) dt ,
\end{align*}

and since $\peps^M \mapsto \epsilon \langle \partial_t \uepsm ,\peps^M \rangle_{\Heps',\Heps} $ is continuous on $\Hepsom$, we get using \cite[Prop. 23.20]{ZeidlerIIA}
\begin{align*}
\langle \partial_t \uepsm , \peps^M \rangle_{(\Hepsom)',\Hepsom} = \epsilon \langle \partial_t \ueps , \peps^M \rangle_{\Heps',\Heps},
\end{align*}
and thus
\begin{align}\label{EstimateRestrictionTimeDerivativeChannels}
\|\partial_t \ueps^M \|_{(\Hepsom)'} \le \epsilon \|\partial_t \ueps\|_{\Heps'}.
\end{align}

\subsection{Existence of a weak solution and a priori estimates}
\label{SectionExistenceAprioriEstimates}

In this section we give the existence result for the microscopic model $\eqref{MicroscopicModel}$ and show a priori estimates depending explicitly on the parameter $\epsilon$. These estimates form the basis for the derivation of the macroscopic model.

\begin{proposition}
\label{ExistenceMicroscopicModel}
There exists a unique weak solution of the microscopic model $\eqref{MicroscopicModel}$.
\end{proposition}
\begin{proof}
The claim follows from the Galerkin-method and a Leray-Schauder fixed-point argument, together with a priori estimates similar as in Lemma \ref{LemmaAprioriEstimates}. Since this is quite standard we skip the details.
\end{proof}

In the next lemma, we state the a priori estimates for the microscopic solution $\ueps$. Therefore, we make use of the following tracee-inequality: For all $\theta >0$ there exists $C(\theta)>0$, such that for all $\peps^M \in H^1(\osem)$ it holds that
\begin{align}\label{TraceInequality}
\|\peps^M\|_{L^2(\neps)} \le \frac{C(\theta)}{\sqrt{\epsilon}} \|\peps^M\|_{L^2(\osem)} + \theta \sqrt{\epsilon}\|\nabla \peps^M\|_{L^2(\osem)}.
\end{align}
This result is easily obtained by decomposing $\osem$ into microscopic cells $\epsilon(Z^{\ast} + (\bar{k},0))$ for $\bar{k} \in I_{\epsilon}$, using a scaling argument, and the usual trace estimate on $\partial Z^{\ast}$.

\begin{lemma}
\label{LemmaAprioriEstimates}
The solution $\ueps$ of the microscopic model $\eqref{MicroscopicModel}$ fulfills the following a priori estimate for a constant $C>0$ independent of $\epsilon$
\begin{align*}
\|\partial_t \ueps\|_{L^2((0,T),\Heps')} + \|\ueps\|_{L^2((0,T),\Heps)} \le C.
\end{align*}
Especially, we obtain
\begin{align*}
\|\partial_t \ueps\|_{L^2((0,T),(\Hepsom)')} \le C\epsilon.
\end{align*}
\end{lemma}
\begin{proof}
We test the variational equation $\eqref{VariationalEquation}$ with $\ueps$ to obtain with the coercivity of $D^M$
\begin{align}
\begin{aligned}\label{LemmaAprioriEstimatesInequality}
\frac{1}{2}\frac{d}{dt}\|\ueps\|^2_{\Leps} & + \sum_{\pm} D^{\pm} \|\nabla \ueps^{\pm}\|^2_{L^2(\oe^{\pm})} + c_0\epsilon \|\nabla \uepsm\|^2_{L^2(\osem)}
\\
&\le  \sum_{\pm} \big(f^{\pm} (\ueps^{\pm}),\ueps^{\pm}\big)_{\oe^{\pm}} + \foe \big(g_{\epsilon}(\uepsm),\uepsm\big)_{\osem} - \big(h_{\epsilon}(\uepsm),\uepsm \big)_{\neps}.
\end{aligned}
\end{align}
We only consider the boundary term on $N_{\epsilon}$ in more detail, since the other terms can be treated in a similar way. The Lipschitz-continuity of $h_{\epsilon}$ and the trace-inequality $\eqref{TraceInequality}$ imply for $\theta >0$ (remember $|N_{\epsilon}|\le C$)
\begin{align*}
\big|\big(h_{\epsilon}(\uepsm),\uepsm \big)_{\neps}\big| &\le C \|1 + \uepsm\|_{L^2(N_{\epsilon})} \|\uepsm\|_{L^2(N_{\epsilon})} \le C \big( 1 + \|\uepsm\|^2_{L^2(N_{\epsilon})}\big)
\\
&\le C(\theta) \left(1 + \foe \|\uepsm\|^2_{L^2(\osem)} \right) + \theta \epsilon \|\nabla \uepsm\|^2_{L^2(\osem)}.
\end{align*}
For $\theta>0$ small enough the last term can be absorbed from the left-hand side in the inequality $\eqref{LemmaAprioriEstimatesInequality}$. Integration with respect to time and the Gronwall inequality implies the estimate for $\ueps$ in the norm of $L^2((0,T),\Heps))$. For the estimate of the time-derivative we just test equation $\eqref{VariationalEquation}$ with $\peps \in \Heps$ and $\|\peps\|_{\Heps} \le 1$:

\begin{align*}
\langle \partial_t \ueps , \peps\rangle_{\Heps',\Heps} =& -\sum_{\pm} \big(D^{\pm} \nabla \ueps^{\pm},\nabla \peps\big)_{\oe^{\pm}} - \epsilon \left(D^M\left(\frac{\cdot}{\epsilon}\right)\nabla \uepsm,\nabla \peps\right)_{\osem}
\\
&+ \sum_{\pm} \big(f^{\pm} (\ueps^{\pm}),\peps\big)_{\oe^{\pm}} + \foe \big(g_{\epsilon}(\uepsm),\peps\big)_{\osem} - \big(h_{\epsilon}(\uepsm),\peps\big)_{\neps}
\\
\le& C \bigg( \sum_{\pm} \Vert \nabla \ueps^{\pm} \Vert_{L^2(\oe^{\pm})} \Vert \nabla \peps \Vert_{L^2(\oe^{\pm})} + \epsilon \Vert \nabla \uepsm \Vert_{L^2(\osem} \Vert \nabla \peps\Vert_{L^2(\osem)} 
\\
& + \sum_{\pm} \big(\Vert \peps\Vert_{L^1(\oe^{\pm})}  + \Vert \ueps^{\pm} \Vert_{L^2(\oe^{\pm})} \Vert \peps \Vert_{L^2(\oe^{\pm})} \big)+ \foe \Vert \peps \Vert_{L^1(\osem)}    
\\
&+ \Vert \uepsm \Vert_{L^2(\osem)} \Vert \peps \Vert_{L^2(\osem)}  + \Vert \peps \Vert_{L^1(\neps)} + \Vert \uepsm \Vert_{L^2(\neps)} \Vert \peps \Vert_{L^2(\neps)}
\bigg)
\\
\le& C \left(1 + \Vert \ueps \Vert_{\Heps} \right) \Vert \peps \Vert_{\Heps}
\\
\le& C \left(1 + \Vert \ueps \Vert_{\Heps} \right),
\end{align*}
where we used similar arguments as above and  again the trace inequality $\eqref{TraceInequality}$. This implies almost everywhere in $(0,T)$
\begin{align*}
\Vert \partial_t \ueps \Vert_{\Heps'} \le C \left(1 + \Vert \ueps \Vert_{\Heps} \right).
\end{align*}

Integration with respect to time and using the estimates already obtained at the beginning of the proof we obtain
\begin{align*}
\Vert \partial_t \ueps \Vert_{L^2((0,T),\Heps')} \le C .
\end{align*}

The last inequality in the Lemma follows from $\eqref{EstimateRestrictionTimeDerivativeChannels}$.
\end{proof}
To pass to the limit in the nonlinear terms of $\eqref{VariationalEquation}$, we need strong two-scale compactness results. Due to the critical scaling of the equation in the channels,  the two-scale limit $u_0^M$ of $\uepsm$ is depending on a macroscopic variable $\x \in \Sigma$ and a microscopic variable $y \in Z^{\ast}$, see Theorem \ref{MainTheoremConvergenceChannels}. Hence, usual compactness arguments, for example a direct application of the Aubin-Lions lemma, fail.  More precisely, extending the functions $\uepsm$ to the whole thin layer $\oem$, such that the a priori estimates from Lemma \ref{LemmaAprioriEstimates} remain valid, and transform $\oem$ to the fixed domain $\Sigma \times (-1,1)$ will not give a uniform bound for the sequence of the gradient with respect to $\epsilon$ and therefore the Aubin-Lions lemma is not applicable. To overcome this problem we use the unfolding operator for thin layers and apply a Kolmogorov-type compactness result to the unfolded sequence. This argument is based on error estimates for the difference of discrete shifts of the microscopic solution. We introduce the following notation:

For an arbitrary set $U\subset \R^n$ and a function $\veps :U\rightarrow \R$, we define for $l\in \Z^n$   the shifted function
\begin{align*}
\veps^l(x):= \veps(x + \epsilon l).
\end{align*} 
Here, if not stated otherwise, we extend $\veps$ by zero to $\R^n$. Then, we define the difference between the shifted function and the function itself by
\begin{align*}
\delta_l \veps(x) := \delta \veps(x) := \veps^l(x) - \veps(x) = \veps(x+\epsilon l) - \veps(x).
\end{align*}
If it is clear from the context we neglect the index $l$ and just write $\delta \veps$. Now, for $0< h \ll 1$ we define
\begin{align*}
\Sigma_h:= \{x \in \Sigma\, : \, \mathrm{dist}(\partial \Sigma,x)>h\},
\end{align*}
and
\begin{align*}
\oeh:= \oe \cap \big(\Sigma_h \times (-H,H)\big), \quad \oeh^{\pm} := \oeh \cap \oe^{\pm}.
\end{align*}
We emphasize that in the definition of $\oeh$ channels can be intersected by the boundary of $\Sigma_h \times (-H,H)$ and we would loose the Lipschitz-regularity of $\oeh$. Therefore we use the domain $\who_{\epsilon,h}$ defined in the following. Let
\begin{align*}
I_{\epsilon,h}:= \big\{\bar{k} \in \Z^{n-1} \, : \, \epsilon(Y +\bar{k} ) \subset \Sigma_h \big\}, \quad \whs_h:= \mathrm{int} \left(\bigcup_{\bar{k}\in I_{\epsilon,h}} \epsilon \big(\overline{Y} + \bar{k}\big)\right),
\end{align*}
and
\begin{align*}
\who_{\ast,\epsilon,h}^M:= \osem\cap \big(\whs_h \times (-\epsilon,\epsilon)\big),\quad \who^+_{\epsilon,h}:= \whs_h \times (\epsilon,H),\quad \who^-_{\epsilon,h}:= \whs_h \times (-H,-\epsilon).
\end{align*}
The top and the bottom of the channels in $\who_{\ast,\epsilon,h}$ are denoted by $\widehat{S}_{\ast,\epsilon,h}^{\pm}$, and the lateral boundary of the channels by $\widehat{N}_{\epsilon,h}$. Then we define
\begin{align*}
\who_{\epsilon,h}:= \who_{\epsilon,h}^+ \cup \who_{\epsilon,h}^- \cup \who_{\ast,\epsilon,h}^M \cup \widehat{S}_{\ast,\epsilon,h}^+ \cup \widehat{S}_{\ast,\epsilon,h}^-.
\end{align*}
On this domain we define the function space
\begin{align*}
\Hepsho := \left\{\veps \in H^1(\who_{\epsilon,h})\, : \, \veps|_{\partial \whs_h \times (-H,-\epsilon)} = \veps|_{\partial \whs_h \times (\epsilon,H)} = 0 \right\} \subset \Heps,
\end{align*}
where we consider $\Hepsho$ as a subset of $\Heps$ by extending functions from $\Hepsho$ by zero to the whole domain $\oe$. In a similar way we define
\begin{align*}
\Lepsh := L^2(\who_{\epsilon,h}) = L^2(\who_{\epsilon,h}^+) \times L^2(\who_{\ast,\epsilon,h}^M ) \times L^2(\who_{\epsilon,h}^-) \subset \Leps,
\end{align*}
what leads to the Gelfand-triple
\begin{align*}
\Hepsho \hookrightarrow \Lepsh \hookrightarrow \Hepsho'.
\end{align*}
We emphasize that on $\Lepsh$ we consider the same inner product as on $\Leps$, just by extending function in $\Lepsh$ by zero to the whole domain $\oe$.

\begin{lemma}\label{LemmaErrorEstimatesShifts}
Let $\ueps$ be the solution of the microscopic model $\eqref{MicroscopicModel}$. Then, for every $0<h\ll 1$, there exists a $C=C(h)>0$, such that for all $l \in \Z^{n-1} \times \{0\}$ with $ |\epsilon l|\ll h$ it holds that
\begin{align}
\begin{aligned}
\label{LemmaErrorEstimatesShiftsInequality}
\frac{1}{\sqrt{\epsilon}}\|\delta \ueps &\|_{L^{\infty}((0,T),L^2(\who_{\ast,\epsilon,2h}^M))} + \sqrt{\epsilon} \|\nabla \delta \ueps \|_{L^2((0,T)\times \who_{\ast,\epsilon,2h}^M)} 
\\
\le & C\left( \epsilon + \|\delta \ueps(0)\|_{\Lepsh} + \sum_{\pm} \|\delta \ueps^{\pm}\|_{L^2((0,T)\times \who_{\epsilon,h}^{\pm})} \right).
\end{aligned}
\end{align}
\end{lemma}
\begin{proof}
We have $\ueps|_{\who_{\epsilon,h}}, \, \ueps^l|_{\who_{\epsilon,h}} \in H^1((0,T),\Hepsho')$ and for all $\pepsh \in \Hepsho$ it holds almost everywhere in $(0,T)$ that
\begin{align*}
\langle \partial_t &\delta \ueps , \pepsh\rangle_{\Hepsho',\Hepsho} \hspace{-2.23pt}+ \sum_{\pm} \big(D^{\pm} \nabla \delta \ueps^{\pm},\nabla \pepsh\big)_{\who_{\epsilon,h}^{\pm}} 
\hspace{-2.23pt}+ \epsilon \left(D^M\left(\fxe\right) \nabla \delta \uepsm , \nabla \pepsh\right)_{\who_{\ast,\epsilon,h}^M}
\\
&= \sum_{\pm} \left(f^{\pm}\big((\ueps^{\pm})^l\big) - f^{\pm}(\ueps^{\pm}) , \pepsh\right)_{\who_{\epsilon,h}^\pm} 
+ \foe \left(g_{\epsilon}\big((\uepsm)^l\big) - g_{\epsilon}(\uepsm),\pepsh\right)_{\who_{\ast,\epsilon,h}^M} 
\\
& - \left(h_{\epsilon}\big((\uepsm)^l\big)- h_{\epsilon}(\uepsm),\pepsh\right)_{\widehat{N}_{\epsilon,h}}.
\end{align*}
Now, let $\eta \in C_0^{\infty}\big(\whs_h\big)$ be a cut-off function with $0\le \eta \le 1$ and $\eta = 1 $ in $\whs_{2h}$, and $\|\nabla_{\x} \eta \|_{L^{\infty}(\whs_h)} \le C=C(h)$. 
As a test-function in the equation above we choose $\pepsh=\eta^2 \delta \ueps$ and obtain with the coercivity of $D^M$
\begin{align*}
\frac12 \frac{d}{dt} \|\eta \delta \ueps & \|^2_{\Lepsh} + \sum_{\pm} D^{\pm} \big\|\eta \nabla \delta \ueps^{\pm}\big\|^2_{L^2(\who_{\epsilon,h}^{\pm})} + c_0 \epsilon \big\|\eta \nabla \delta \uepsm\big\|^2_{L^2(\who_{\ast,\epsilon,h}^M)}
\\
\le & \sum_{\pm} \left\{-2 \left(D^{\pm}\nabla \delta \ueps^{\pm}, \eta \delta \ueps^{\pm} \nabla \eta\right)_{\who_{\epsilon,h}^{\pm}} + \left(\delta f^{\pm}(\ueps^{\pm}), \eta^2 \delta \ueps^{\pm}\right)_{\who_{\epsilon,h}^{\pm}} \right\}
\\
&-2 \epsilon \left(D^M\left(\fxe\right) \nabla \delta \uepsm , \eta \delta \uepsm \nabla \eta \right)_{\who_{\ast,\epsilon,h}^M} + \foe \left(\delta g_{\epsilon}(\uepsm) , \eta^2 \delta \uepsm \right)_{\who_{\ast,\epsilon,h}^M} 
\\
&- \left(\delta h_{\epsilon} (\uepsm) ,\eta^2 \delta \uepsm\right)_{\widehat{N}_{\epsilon,h}}
\\
=&:  \sum_{\pm} \left[I_{\epsilon,1}^{\pm} + I_{\epsilon,2}^{\pm}\right] + \sum_{j=3}^5 I_{\epsilon,j}^M.
\end{align*}
Integration with respect to time gives us for almost every $t \in (0,T)$
\begin{align*}
 \|\eta \delta \ueps(t) &\|^2_{\Lepsh} + \sum_{\pm} \|\eta \nabla \delta  \ueps^{\pm} \|^2_{L^2((0,t)\times \who_{\epsilon,h}^{\pm})} +  \epsilon \|\eta \nabla \delta \uepsm\|^2_{L^2((0,t)\times \who_{\ast,\epsilon,h}^M)}
\\
&\le C \int_0^t \left\{\sum_{\pm} \left[I_{\epsilon,1}^{\pm} + I_{\epsilon,2}^{\pm}\right] + \sum_{j=3}^5 I_{\epsilon,j}^M\right\} dt + \sum_{\pm} I_{\epsilon,6}^{\pm} + I_{\epsilon,7}^M,
\end{align*}
with 
\begin{align*}
I_{\epsilon,6}^{\pm} := \|\eta \delta \ueps^{\pm}(0)\|^2_{L^2(\who_{\epsilon,h}^{\pm})}, \quad I_{\epsilon,7}^M:= \foe \|\eta \delta \uepsm (0)\|^2_{L^2(\who_{\ast,\epsilon,h}^M)}.
\end{align*}
We start to estimate the term $I_{\epsilon,1}^{\pm}$. For $\theta >0$ we obtain
\begin{align*}
\int_0^t I_{\epsilon,1}^{\pm} dt &\le C\|\eta \nabla \delta \ueps^{\pm}\|_{L^2((0,t)\times\who_{\epsilon,h}^{\pm})} \|\delta \ueps^{\pm}\|_{L^2((0,t)\times \who_{\epsilon,h}^{\pm})} 
\\
&\le C(\theta) \|\delta \ueps^{\pm}\|^2_{L^2((0,t)\times\who_{\epsilon,h}^{\pm})}  + \theta \|\eta \nabla \delta \ueps^{\pm}\|^2_{L^2((0,t)\times\who_{\epsilon,h}^{\pm})},
\end{align*}
for a constant $C(\theta)>0$ depending on $\theta$. In a similar way and using the a priori estimate from Lemma \ref{LemmaAprioriEstimates}, we obtain for $I_{\epsilon,3}^M$
\begin{align*}
\int_0^t I_{\epsilon,3}^M dt &\le C \epsilon \|\nabla \delta \uepsm \|_{L^2((0,t)\times \who_{\ast,\epsilon,h}^M)} \|\eta  \delta \uepsm\|_{L^2((0,t)\times \who_{\ast,\epsilon,h}^M)}
\\
&\le C \left(\epsilon^3 \|\nabla \delta \uepsm\|^2_{L^2((0,t)\times \who_{\ast,\epsilon,h}^M)} + \foe \|\eta \delta \uepsm\|^2_{L^2((0,t)\times \who_{\ast,\epsilon,h}^M)}\right)
\\
&\le C \left( \epsilon^2 + \foe \|\eta \delta \uepsm\|^2_{L^2((0,t)\times \who_{\ast,\epsilon,h}^M)}\right).
\end{align*}
For $I_{\epsilon,5}^M$ we use the Lipschitz-continuity of $h_{\epsilon}$, the trace-inequality $\eqref{TraceInequality}$, and again the a priori estimate from Lemma \ref{LemmaAprioriEstimates} to obtain for $\theta >0$
\begin{align*}
&\int_0^t  I_{\epsilon,5}^M dt \le C \|\eta \delta \uepsm\|^2_{L^2((0,t)\times \widehat{N}_{\epsilon,h})} 
\\
&\le  \frac{C(\theta)}{\epsilon} \|\eta \delta \uepsm\|^2_{L^2((0,t)\times \who_{\ast,\epsilon,h}^M)} + C\epsilon \|\delta \uepsm\|^2_{L^2((0,t)\times \who_{\ast,\epsilon,h}^M)} + \theta \epsilon \|\eta \nabla \delta \uepsm\|^2_{L^2((0,t)\times \who_{\ast,\epsilon,h}^M)}
\\
&\le \frac{C(\theta)}{\epsilon} \|\eta \delta \uepsm\|^2_{L^2((0,t)\times \who_{\ast,\epsilon,h}^M)} + C\epsilon^2 + \theta \epsilon \|\eta \nabla \delta \uepsm\|^2_{L^2((0,t)\times \who_{\ast,\epsilon,h}^M)}.
\end{align*}
For $I_{\epsilon,2}^{\pm}$ and $I_{\epsilon,4}^M$ we obtain directly from the Lipschitz-continuity of $f^{\pm}$ and $g_{\epsilon}$
\begin{align*}
\int_0^t \sum_{\pm} I_{\epsilon,2}^{\pm} + I_{\epsilon,4}^M dt \le C \|\delta \ueps\|^2_{L^2((0,t),\Lepsh)}.
\end{align*}
Choosing $\theta>0$ small enough, the desired result follows from the Gronwall-inequality.
\end{proof}

\begin{remark}
The error term $I_{\epsilon,1}^{\pm}$ arises due to the cut-off function $\eta$ in the proof of Lemma \ref{LemmaErrorEstimatesShifts}. Hence, with the method used above we are not able to get rid of the norms of $\delta \ueps^{\pm}$ on the right-hand side of inequality $\eqref{LemmaErrorEstimatesShiftsInequality}$. For specific boundary conditions, like   zero Dirichlet-boundary conditions, or in case of a rectangle $\Sigma$, zero Neumann-boundary conditions  or periodic boundary conditions on the lateral boundary, it is easily possible to extend the solution $\ueps$ in $\x$-direction and obtain an estimate of the form
\begin{align*}
\|\delta \ueps\|_{L^2((0,T),\Heps)} \le C \left(\epsilon + \|\delta \ueps(0)\|_{\Leps}\right).
\end{align*}
However, the interior estimate in Lemma \ref{LemmaErrorEstimatesShifts} is enough 
to obtain strong two-scale compactness results in the channel domain and the method presented in the proof has the advantage that it is applicable for more general boundary conditions.
\end{remark}

\section{Two-scale convergence and the unfolding operator for thin channels}
\label{Sect_Two_scale_convergence}

In this section we define the two-scale convergence for thin channels and give some weak two-scale compactness results based on a priori estimates in $L^2((0,T),\Heps)$. Further, we derive strong two-scale convergence results based on error estimates for the discrete shifts as in Lemma \ref{LemmaErrorEstimatesShifts}. Therefore, we make use of the unfolding operator and a Kolmogorov-type compactness result. We start with the definition of the two-scale convergence for channels, see also \cite{BhattacharyaGahnNeussRadu}.

\begin{definition}
Let $p\in [1,\infty)$ and $p'$ the dual exponent of $p$. 
\begin{enumerate}
[label = (\roman*)]
\item  We say the sequence $\veps \in L^p((0,T) \times \osem)$ converges (weakly) in the two-scale sense to a limit function $v_0\in  L^p((0,T) \times \Sigma \times Z^{\ast})$, if 
\begin{align*}
\lim_{\epsilon \to 0} \frac{1}{\epsilon} \int_{0}^{T} \int _{\osem}& \veps(t,x) \psi \left(t,\x,\frac{x}{\epsilon}\right) \,dx \,dt
\\
&= \int_{0}^{T} \int _{\Sigma} \int _{Z^{\ast}} v_0(t,\x,y)\psi(t,\x,y) \,dy \,d\x \,dt,
\end{align*} 
for all $\psi \in  L^{p'}((0,T) \times \Sigma, C_{\per}( \overline{Z^{\ast}}))$. The sequence converges strongly in the two-scale sense (in $L^p$) if it holds that
\begin{align*}
\lim_{\epsilon \to 0} \epsilon^{-\frac{1}{p}} \|\veps\|_{L^p((0,T)\times \osem)} = \|v_0\|_{L^p((0,T)\times \Sigma \times Z^{\ast})}.
\end{align*}
\item We say the sequence $\veps \in L^p((0,T) \times \neps)$ converges (weakly) in the two-scale sense to a limit function $v_0(t,\x,y) \in  L^p((0,T) \times \Sigma \times N)$ on $\neps$, if  
\begin{align*}
\lim_{\epsilon \to 0} \int_{0}^{T} \int _{\neps} \hspace{-0.31em} \veps(t,x) \psi \left(t,\x,\frac{x}{\epsilon}\right) \,d\sigma \,dt= \int_{0}^{T} \int _{\Sigma} \int _{N} v_0(t,\x,y)\psi(t,\x,y) \,d\sigma_y \,d\x \,dt,
\end{align*} 
for all $\psi \in  L^{p'}((0,T), C(\overline{\Sigma}, C_{\per}( \overline{N})))$. The sequence converges strongly in the two-scale sense (in $L^p$) on $\neps$ if it holds that
\begin{align*}
\lim_{\epsilon \to 0} \|\veps\|_{L^p((0,T)\times \neps)} = \|v_0 \|_{L^p((0,T)\times \Sigma \times N)}.
\end{align*}
\end{enumerate}
We just say a sequence converges in the two-scale sense, if it converges in the two-scale sense in $L^p$.
\end{definition}

In the following Lemma we give some weak two-scale compactness results in the microscopic channels, based on a priori estimates of the microscopic solution.

\begin{lemma}\label{LemmaTSWeakCompactness}
Let $p\in (1,\infty)$ and $p'$ the dual exponent of $p$.
\begin{enumerate}[label = (\roman*)]
\item\label{TSWeakCompactnessBasicConvergence} Let $\veps$ be a sequence of functions in $L^p((0,T) \times \osem)$ such that 
\begin{align*} 
\epsilon^{-\frac{1}{p}} \| \veps \|_{L^p((0,T) \times \osem)} \leq C. 
\end{align*}
 Then, there exists  $v_0 \in  L^p((0,T) \times \Sigma \times Z^{\ast})$ such that, up to a subsequence, $\veps$ two-scale converges to $v_0$.
\item\label{TSCompactnessGradient} Let $\veps $  be a sequence of functions in $ L^p((0,T), W^{1,p}(\osem))$ such that
\begin{align*}
\epsilon^{-\frac{1}{p}} \|\veps \|_{L^p((0,T)\times \osem)} + \epsilon^{\frac{1}{p'}} \|\nabla \veps\|_{L^p((0,T)\times \osem)} \le C.
\end{align*}
Then, there exists $v_0 \in L^p((0,T)\times \Sigma , W^{1,p}(Z^{\ast}))$, such that up to a subsequence $\veps \rightarrow v_0$ and $\epsilon\nabla \veps \rightarrow \nabla_y v_0$ in the two-scale sense.
\item\label{TSWeakCompactnessBasicConvergenceSurface} Let $\veps $ be a sequence of functions in $L^p((0,T)\times \neps)$ such that
\begin{align*}
\|\veps\|_{L^p((0,T)\times \neps)} \le C.
\end{align*}
Then, there exists $v_0 \in L^p((0,T)\times \Sigma \times N)$, such that $\veps \rightarrow v_0$ in the two-scale sense on $\neps$.
\end{enumerate}
\end{lemma}
\begin{proof}
Statements \ref{TSWeakCompactnessBasicConvergence} and \ref{TSWeakCompactnessBasicConvergenceSurface} were shown in \cite[Theorem 4.4]{BhattacharyaGahnNeussRadu} for $p=2$. The general case works the same lines. To prove \ref{TSCompactnessGradient} we notice that 
\begin{align*}
\epsilon^{-\frac{1}{p}} \big(\|\veps\|_{L^p((0,T) \times \osem)} + \|\epsilon\nabla \veps\|_{L^p((0,T) \times \osem)}\big) \le C.
\end{align*}
Hence, there exist $v_0 \in L^p((0,T)\times \Sigma \times Z^{\ast})$ and $\xi_0 \in L^p((0,T)\times \Sigma \times Z^{\ast})^n$, such that up to a subsequence
\begin{align*}
\veps &\rightarrow v_0 &\mbox{ in the two-scale sense}&,
\\
\epsilon \nabla \veps &\rightarrow \xi_0 &\mbox{ in the two-scale sense}&.
\end{align*}
By integration by parts we obtain for all $\Phi \in C^{\infty}_0((0,T)\times \Sigma \times Z^{\ast})^n$
\begin{align*}
\int_0^T \int_{\Sigma} \int_{Z^{\ast}} v_0 &\nabla_y \cdot \Phi (t,\x,y) dy d\x dt 
\\
&= \lim_{\epsilon\to 0} \int_0^T \int_{\osem}\veps \left[ \epsilon \nabla_{\x} \cdot \Phi\left(t,\x,\fxe\right) + \nabla_y \cdot \Phi\left(t,\x,\fxe\right) \right] dx dt
\\
&= - \lim_{\epsilon\to 0} \epsilon \int_0^T\int_{\osem} \nabla \veps \cdot \Phi\left(t,\x,\fxe\right) dx dt 
\\
&= -\int_0^T \int_{\Sigma} \int_{Z^{\ast}} \xi_0 \cdot \Phi (t,\x,y)dy d\x dt,
\end{align*}
which yields the desired result.
\end{proof}

Lemma \ref{LemmaTSWeakCompactness} and the a priori estimates for the microscopic solution from Lemma \ref{LemmaAprioriEstimates} are enough to pass to the limit in the linear terms in the channels $\osem$ in the variational equation $\eqref{VariationalEquation}$. To pass to the limit in the nonlinear terms, we need strong two-scale convergence. To establish these convergence results, we use the unfolding operator for channels defined below. This operator is closely related to the unfolding operator in thin domains, see for example \cite{NeussJaeger_EffectiveTransmission} and \cite{CioranescuDamlamianGrisoOnofrei2008}.

\begin{definition}
Let $(G_{\epsilon},G) \in \{(\osem,Z^{\ast}),(\neps,N)\}$. Then for $p\in[1,\infty)$ we define the unfolding operator
\begin{align*}
\te : L^p((0,T)\times G_{\epsilon}) \rightarrow L^p((0,T) \times \Sigma \times G),
\\
\te \veps(t,\x,y)= \veps\left(t,\epsilon\left(\left[\frac{\x}{\epsilon}\right],0\right) + \epsilon y\right),
\end{align*}
where $[\cdot]$ denotes the integer part of $\cdot$.
\end{definition}
Here, for an easier notation, we use the same notation for the unfolding operator for different domains of definition as for the usual unfolding operator for domains defined in \cite{Cioranescu_Unfolding1}. It should be clear from the context in which sense it has to be understood.  Further it makes sense to use the same notation for the unfolding operator on $\osem$ and $\neps$, since the unfolding operator commutes with the trace operator in the following sense: For $\veps \in L^p((0,T),W^{1,p}(\osem))$ it holds that
\begin{align*}
\te (\veps)|_{N} = \te\big(\veps|_{\neps}\big).
\end{align*}
We summarize some properties of $\te$:
\begin{lemma}\label{LemmaPropertiesUnfoldingOperator}
Let $p\in [1,\infty)$. Then it holds that:
\begin{enumerate}
[label = (\roman*)]
\item For $\veps,  \in L^p((0,T) \times \osem)$ and $\weps \in L^{p'}((0,T)\times \osem)$, where $p'$ denotes the dual exponent of $p$, we have
\begin{align*}
(\te \veps , \te \weps)_{(0,T)\times \Sigma \times Z^{\ast}} &= \foe(\veps,\weps)_{(0,T)\times \osem},
\\
\|\te \veps \|_{L^p((0,T)\times \Sigma \times Z^{\ast})} &= \epsilon^{-\frac{1}{p}} \|\veps \|_{L^p((0,T)\times \osem)},
\end{align*}
and for $\veps \in L^p((0,T),W^{1,p}(\osem))$ it holds that
\begin{align*}
\nabla_y \te \veps = \epsilon \te \nabla \veps.
\end{align*}
\item For $\veps \in L^p((0,T)\times \neps)$ and $\weps \in L^{p'}((0,T)\times \neps)$, where $p'$ denotes the dual exponent of $p$, we have
\begin{align*}
(\te \veps, \te \weps)_{(0,T)\times \Sigma \times N} &= (\veps,\weps)_{(0,T)\times \neps},
\\
\|\te \veps \|_{L^p((0,T)\times \Sigma \times N )} &= \|\veps\|_{L^p((0,T)\times \neps)}.
\end{align*}
\end{enumerate}
\end{lemma}
\begin{proof}
These results are obtained by a simple calculation. For the main ideas of the proof see \cite{Cioranescu_Unfolding1}.
\end{proof}
We have the following relation between the two-scale convergence and the unfolding operator.

\begin{lemma}\mbox{}
\label{LemmaAequivalenzUnfoldingTSConvergence}
\begin{enumerate}
[label = (\roman*)]
\item The sequence $\veps \in L^p((0,T)\times \osem)$ for $p \in (1,\infty)$ converges weakly/strongly in the two-scale sense in $L^p$, if and only if $\te \veps$ converges weakly/strongly in $L^p((0,T)\times \Sigma \times Z^{\ast})$ to the same limit.
\item The sequence $\veps \in L^p((0,T)\times \neps)$ for $p \in (1,\infty)$ converges weakly/strongly in the two-scale sense in $L^p$, if and only if $\te \veps$ converges weakly/strongly in $L^p((0,T)\times \Sigma \times N)$ to the same limit.
\end{enumerate}
\end{lemma}
\begin{proof}
This result was obtained for $p=2$ and weak convergences in \cite{BourgeatLuckhausMikelic} for bulk domains. However, the proof can easily be extended to our setting. The strong convergence results follow from the properties of the unfolding operator in Lemma \ref{LemmaPropertiesUnfoldingOperator}.
\end{proof}

The following Lemma shows, that the strong two-scale convergence is sufficient to pass to the limit in the nonlinear terms.

\begin{lemma}\label{LemmaTSStrongConvergence}
Let $g_{\epsilon}$ and $h_{\epsilon}$ satisfy the Assumptions \ref{AssumptionReactionRateChannels} and \ref{AssumptionReactionRateSurfaceChannels}. 
\begin{enumerate}
[label = (\roman*)]
\item Let $\veps $ be a sequence in $L^2((0,T)\times \osem)$ such that
\begin{align*}
\frac{1}{\sqrt{\epsilon}} \|\veps\|_{L^2((0,T)\times \osem)} \le C,
\end{align*}
and $\veps$ converges to $v_0 \in L^2((0,T)\times \Sigma \times Z^{\ast})$ strongly in the two-scale sense in $L^p$ for $p\in [1,2]$. Then it holds that
\begin{align*}
g_{\epsilon}(\veps) \rightarrow g(v_0) \quad \mbox{ in the two-scale sense}.
\end{align*}
\item Let $\veps $ be a sequence in $L^2((0,T)\times \neps)$ such that
\begin{align*}
\|\veps\|_{L^2((0,T)\times \neps)} \le C,
\end{align*}
and $\veps$ converges to $v_0 \in L^2((0,T)\times \Sigma \times N)$ strongly in the two-scale sense in $L^p$ on $\neps$ for $p\in [1,2]$. Then it holds that
\begin{align*}
h_{\epsilon}(\veps) \rightarrow h(v_0) \quad \mbox{ in the two-scale sense}.
\end{align*}
\end{enumerate}

\end{lemma}
\begin{proof}
We only prove the second result, since the first one follows in a similar way. Let $\phi  \in L^{p'}((0,T),C(\overline{\Sigma} ,C_{\per}(\overline{N})))$,  $v^n \in C_0^{\infty}((0,T)\times \Sigma, C_{\per}(\overline{N}))$ for $n\in \N$ such that $v^n \rightarrow v_0 $ in $L^2((0,T)\times \Sigma \times N)$, and we define $\veps^n(t,x):= v^n\left(t,\x,\fxe\right)$.

\begin{align*}
\int_0^T \int_{\neps} h_{\epsilon}(\veps) \phi\left(t,\x,\fxe\right)d\sigma dt &= \int_0^T \int_{\neps} \big[h_{\epsilon}(\veps) - h_{\epsilon}(\veps^n)\big] \phi\left(t,\x,\fxe\right) d\sigma dt 
\\
&\hspace{2em} + \int_0^T \int_{\neps} h_{\epsilon}(\veps^n) \phi\left(t,\x,\fxe\right) d\sigma dt =: I_{\epsilon,1} + I_{\epsilon,2}.
\end{align*}
For the first term $I_{\epsilon,1}$ we use the Lipschitz continuity of $h_{\epsilon}$ and the properties of the unfolding operator to obtain (with $\te \veps^n(t,\x,y)=v^n\left(t,\epsilon \left[\frac{\x}{\epsilon}\right] + \epsilon \bar{y} ,y\right)$)
\begin{align*}
|I_{\epsilon,1}|&\le C\|\veps - \veps^n\|_{L^p((0,T)\times \neps)} =  C \|\te \veps - \te \veps^n\|_{L^p((0,T)\times \Sigma \times N)}
\\
&\le C\big(\|\te \veps - v_0\|_{L^p((0,T)\times \Sigma \times N)} +  \|v_0 -  v^n\|_{L^p((0,T)\times \Sigma \times N)} 
\\
&\hspace{3em}+ \|v^n - \te \veps^n\|_{L^p((0,T)\times \Sigma \times N)}\big).
\end{align*}
The first term convergence to zero for $\epsilon \to 0$, due to the strong two-scale convergence of $\veps$ and Lemma \ref{LemmaAequivalenzUnfoldingTSConvergence}. The second term goes to zero for $n\to \infty$, and the last term vanishes for $\epsilon \to 0$, due to the dominated convergence theorem of Lebesgue, since $\te \veps^n \rightarrow v^n$ almost everywhere in $(0,T)\times \Sigma \times N$.
Let us estimate $I_{\epsilon,2}$:
\begin{align*}
I_{\epsilon,2} =& \bigg[\int_0^T\int_{\neps} h_{\epsilon}(\veps^n) \phi \left(t,\x,\fxe\right) d\sigma dt - \int_0^T \int_{\Sigma} \int_N h(v^n) \phi(t,\x,y) d\sigma_y d\x dt \bigg]
\\
&+ \int_0^T \int_{\Sigma} \int_N \big[ h(v^n) - h(v_0)\big] \phi d\sigma_y d\x dt
+ \int_0^T \int_{\Sigma} \int_N h(v_0) \phi d\sigma_y d\x dt.
\end{align*}
The term in the brackets converges to zero for $\epsilon \to 0$, due to the oscillation lemma, see \cite[Lemma 4.3]{BhattacharyaGahnNeussRadu}. The second term vanishes for $n \to \infty$, due to the Lipschitz continuity of $h$. This gives the desired result.
\end{proof}

To establish the strong two-scale convergence of $\uepsm$ we will show the strong convergence in $L^p$ of $\te \uepsm$. This requires to control the dependence on the time-variable. Since the time-derivative of $\ueps$ respectively $\uepsm$ only exists in a weak sense, in fact we have $\partial_t \uepsm \in L^2((0,T),(\Hepsom)')$, it is not obvious in which space $\partial_t \te \uepsm$ lies and how its norm can be estimated with respect to $\epsilon$. To overcome this problem we use a functional analytical argument. We consider the  $L^2$-adjoint of $\te$, the so called averaging operator $\ue$, to obtain a representation of $\partial_t \te$ via the averaging operator. Therefore, we have to restrict the domain of definition for $\te$ and $\ue$. This idea was already used in \cite{GahnNeussRaduKnabner2018a} and here we put in a more general framework. First of all, let us give a general functional analytic result:
\begin{lemma}\label{LemmaGeneralTimeDerivative}
Let $V,\, W $ be reflexive, separable Banach-spaces, and $Y,\, X$ Hilbert-spaces,  such that we have the Gelfand-triples
\begin{align*}
V \hookrightarrow Y \hookrightarrow V', \quad W \hookrightarrow X \hookrightarrow W',
\end{align*}
with continuous and dense embeddings. Here, we identify $Y$ and $X$ with their dual spaces $Y'$ and $X'$ via the Riesz-representation theorem. Let $A \in \mathcal{L}(Y,X)$ and we denote by $A^{\ast} \in \mathcal{L}(X,Y)$ the adjoint operator of $A$. If $A^{\ast}(W) \subset V$ with $\|A^{\ast}w\|_V \le C \|w\|_W$ for all $w \in W$, and $u \in L^2((0,T),Y)\cap H^1((0,T),V')$, then it holds $\partial_t Au \in L^2((0,T),W')$ with
\begin{align*}
\langle \partial_t Au , w \rangle_{W',W} = \langle \partial_t u, A^{\ast} w\rangle_{V',V} \quad \mbox{ for all } w \in W.
\end{align*}
Here we apply the operator $A$ pointwise to $u$ with respect to $t \in (0,T)$.
\end{lemma}
\begin{proof}
This is just a consequence of the definition of the adjoint operator and the generalized time-derivative. 
In fact, we have to show, see \cite[Prop. 23.20]{ZeidlerIIA}, that there exists $F \in L^2((0,T),W')$ such that for all $w \in W$ and $\psi \in \mathcal{D}(0,T)$ it holds that 

\begin{align*}
-\int_0^T \langle F, w \rangle_{W',W} \psi dt =  \int_0^T (Au,w)_X \psi' dt.
\end{align*}

We have for all $w\in W$ and $\psi \in \mathcal{D}(0,T)$
\begin{align*}
\int_0^T (Au, w)_X \psi'dt = \int_0^T (u,A^{\ast}w)_Y \psi' dt  = -\int_0^T \langle \partial_t u , A^{\ast} w\rangle_{V',V} \psi dt.
\end{align*}
Due to our assumptions, we have almost everywhere in $(0,T)$
\begin{align*}
\vert \langle \partial_t u , A^{\ast} w \rangle_{V',V} \vert \le \Vert \partial_t u \Vert_{V'} \Vert A^{\ast} w\Vert_V \le C \Vert \partial_t u \Vert_{V'}\Vert w \Vert_W.
\end{align*}
Hence, $w \mapsto - \langle \partial_t u , A^{\ast} w\rangle_{V',V}  \in L^2((0,T),W')$, which gives the desired result.
\end{proof}

Let us define $\epsilon^{-1} \ue$ as the $L^2$-adjoint of $\te$, \ie let 
\begin{align*}
\ue : L^2((0,T)\times \Sigma \times Z^{\ast}) \rightarrow L^2((0,T) \times \osem),
\end{align*}
such that
\begin{align*}
\big(\te \veps , \phi \big)_{(0,T)\times \Sigma \times Z^{\ast}} = \foe \big(\veps , \ue \phi \big)_{(0,T)\times \osem},
\end{align*}
for all $\veps \in L^2((0,T) \times \osem)$ and $\phi \in L^2((0,T)\times \Sigma \times Z^{\ast})$. It is easy to check that
\begin{align*}
\ue (\phi)(t,x) = \int_Y \phi \left(t,\epsilon \left(\bar{z} + \left[\frac{\x}{\epsilon}\right] \right), \left(\left\{\frac{\x}{\epsilon}\right\},\frac{x_n}{\epsilon}\right)\right) d\bar{z} \quad \mbox{ for } (t,x)\in (0,T)\times \osem,
\end{align*}
and $x = [x] + \{x\}$, but we will not use this explicit formula for $\ue (\phi)$. 
\begin{corollary}\label{KorollarNormAveragingOperator}
For all $\phi \in L^2((0,T)\times \Sigma \times Z^{\ast})$ it holds that
\begin{align*}
\|\ue \phi \|_{L^2((0,T)\times \osem)} \le \sqrt{\epsilon} \|\phi\|_{L^2((0,T)\times \Sigma \times Z^{\ast})}.
\end{align*}
\end{corollary}
\begin{proof}
This follows by a simple duality argument, see also \cite[Corollary 2.15]{GahnDissertation} for more details.
\end{proof}
Concerning the regularity of $\ue(\phi)$ with respect to the spatial variable, we have that
\begin{align*}
\ue : L^2((0,T)\times \Sigma , H^1(Z^{\ast})) \rightarrow L^2(0,T), H^1(\osem))
\end{align*}
with 
\begin{align}
\label{GradientAveragingOperator}
\epsilon \nabla \ue (\phi) = \ue (\nabla_y \phi) \quad \mbox{ for all } \phi \in L^2((0,T)\times \Sigma, H^1(Z^{\ast})).
\end{align}
This result uses the fact that $Z^{\ast}$ is not touching the lateral boundary of $Z$  and can be shown by similar arguments like in the proof of \cite[Proposition 6]{GahnNeussRaduKnabner2018a}.
We emphasize that the situation gets more delicate if the channel $Z^{\ast}$ touches the lateral boundary of $Z$ and in that case
one has to restrict to function spaces with vanishing traces on $\partial Z$, see also \cite{GahnNeussRaduKnabner2018a}.

Next, we apply Lemma \ref{LemmaGeneralTimeDerivative} to obtain a representation of $\partial_t \te \uepsm$ by means of the $\partial_t \uepsm$ and $ \ue$. However, since we have just $\partial_t \uepsm \in L^2((0,T),(\Hepsom)')$, we have to restrict the operator $\ue$. We define
\begin{align*}
\Ho := \{v \in H^1(Z^{\ast}) \, : \, v|_{S_{\ast}^{\pm}} = 0\}\subset H^1(Z^{\ast}).
\end{align*}
and consider
\begin{align*}
\ue : L^2((0,T)\times \Sigma , \Ho) \rightarrow L^2((0,T), \Hepsom).
\end{align*}

\begin{remark}\label{AveragingRemarkTimeIndependent}
We emphasize that in the definitions of the unfolding operator $\te$ and the averaging operator $\ue$ the time-variable acts as an additional parameter. More precisely, both operators may be defined for the time-independent case and then for time-dependent spaces pointwise with respect to $t \in (0,T)$. Hence, in the following we use the same notation for the unfolding operator $\te$ as an operator on $L^2((0,T)\times \osem)$ and $L^2(\osem)$, and in the same way we proceed for the averaging operator $\ue$.
Especially,  Corollary \ref{KorollarNormAveragingOperator} and equation $\eqref{GradientAveragingOperator}$  also hold for time-independent functions.
\end{remark}

\begin{proposition}\label{PropositionExistenceTimeDerivativeUnfoldedSequence}
Let $\veps \in L^2((0,T),L^2(\osem)) \cap H^1((0,T),(\Hepsom)')  $. Then we have $\te \veps \in H^1((0,T),L^2(\Sigma,\Ho)')$ with
\begin{align}\label{RepresentationTimeDerivative}
\langle \partial_t \te \veps(t) , \phi \rangle_{L^2(\Sigma,\Ho)',L^2(\Sigma,\Ho)} = \foe \langle \partial_t \veps(t), \ue \phi \rangle_{(\Hepsom)',\Hepsom}
\end{align}
for all $\phi \in L^2(\Sigma,\Ho)$ and almost every $t \in (0,T)$. Additionally, we have
\begin{align}
\label{InequalityTimeDerivativeUnfoldingOperator}
\|\partial_t \te \veps \|_{L^2((0,T),L^2(\Sigma,\Ho)')} \le \foe \|\partial_t \veps \|_{L^2((0,T),(\Hepsom)')} .
\end{align}
\end{proposition}
\begin{proof}
In Lemma \ref{LemmaGeneralTimeDerivative} we choose (here we first consider $\te$ as a stationary operator, see Remark \ref{AveragingRemarkTimeIndependent}) :
\begin{align*}
V=\Hepsom, \quad Y=L^2(\osem), \quad W=L^2(\Sigma,\Ho), \quad X= L^2(\Sigma \times Z^{\ast}), \quad A=\te.
\end{align*}

Then we have $\te \in \mathcal{L} (Y,X)$ and $A^{\ast} = \epsilon^{-1} \ue \in \mathcal{L}(X,Y)$. Hence, the conditions of Lemma \ref{LemmaGeneralTimeDerivative} are fulfilled and we obtain $\partial_t \te \veps \in L^2((0,T),L^2(\Sigma,\Ho)')$ with $\eqref{RepresentationTimeDerivative}$.
For the estimate $\eqref{InequalityTimeDerivativeUnfoldingOperator}$ we choose $\phi \in L^2(\Sigma,\Ho)$ with $\|\phi\|_{L^2(\Sigma,\Ho)} \le 1$ and obtain
\begin{align*}
\langle \partial_t \te \veps , \phi \rangle_{L^2(\Sigma,\Ho)',L^2(\Sigma,\Ho)} &= \foe \langle \partial_t \veps ,\ue \phi \rangle_{(\Hepsom)',\Hepsom}
\\
&\le \foe \|\partial_t \veps \|_{(\Hepsom)'}\|\ue \phi\|_{\Hepsom} \le \foe \|\partial_t \veps \|_{(\Hepsom)'},
\end{align*}
where the last inequality follows from $\eqref{GradientAveragingOperator}$ and Corollary \ref{KorollarNormAveragingOperator} (see also Remark \ref{AveragingRemarkTimeIndependent}).
\end{proof}

\begin{lemma}
\label{LemmaEstimateShiftsUnfolded}
For all $\peps \in L^2((0,T)\times \osem)$,  $0<h\ll 1$, and $\bxi \in \R^{n-1}$ with $|\bxi| \ll h$, it holds for $\epsilon$ small enough that
\begin{align*}
\big\|\te \peps(\cdot,\cdot + \bxi , \cdot) - \te \peps \big\|_{L^2((0,T)\times \Sigma_{2h} \times Z)}^2 \le \frac{1}{\epsilon} \sum_{\bar{j}\in \{0,1\}^{n-1}} \|\delta_l \peps\|_{L^2((0,T)\times \who_{\ast,\epsilon,h}^M)}^2
\end{align*}
 with $l = l(\epsilon,\bxi,\bar{j}) = \bar{j} + \left[\frac{\bxi}{\epsilon}\right]$.
\end{lemma}
\begin{proof}
The idea of the proof can be found in \cite[page 709]{NeussJaeger_EffectiveTransmission} for a thin layer and can be extended in an obvious way to our setting.
\end{proof}

In the next theorem we formulate a general strong two-scale compactness result for sequences $\veps \in L^2((0,T),H^1(\osem))$  and their traces $\veps|_{\neps}$. This result allows us to pass to the limit in the nonlinear terms in $\eqref{VariationalEquation}$. 
Similar ideas have been used in \cite[Theorem 7.5]{GahnNeussRaduKnabner2018a}, where however, they were carried-out for the sequence of solutions of a microscopic problem in a thin layer with oscillating diffusion coefficients.
%

\begin{theorem}
\label{TheoremStrongTSConvergence}
Let $\veps \in L^2((0,T),H^1(\osem))\cap H^1((0,T),(\Hepsom)')$ such that
\begin{enumerate}
[label = (\roman*)]
\item\label{TheoremStrongTSConvergenceAprioriEstimates} we have the estimate
\begin{align*}
\foe \|\partial_t \veps\|_{L^2((0,T) , (\Hepsom)')} + \frac{1}{\sqrt{\epsilon}} \|\veps\|_{L^2((0,T)\times \osem)} + \sqrt{\epsilon}\|\nabla\veps\|_{L^2((0,T)\times \osem)} \le C,
\end{align*}
\item\label{TheoremStrongTSConvergenceControlShifts} for all $0< h\ll 1$ and $l\in \Z^{n-1}\times \{0\}$ and $|\epsilon l|\ll h$ it holds that
\begin{align*}
\frac{1}{\sqrt{\epsilon}} \|\delta \veps \|_{L^2((0,T)\times \who_{\ast,\epsilon,h}^M)} + \sqrt{\epsilon}\|\nabla \delta \veps \|_{L^2((0,T)\times \who_{\ast,\epsilon,h}^M)} \overset{\epsilon l \to 0 }{\longrightarrow} 0.
\end{align*}
\end{enumerate}
Then there exists $v_0 \in L^2((0,T)\times \Sigma , H^1(Z^{\ast}))$ such that up to a subsequence it holds for $p\in [1,2)$ and $\beta \in \left(\frac12,1\right)$
\begin{align*}
\veps &\rightarrow v_0 &\mbox{ in the two-scale sense},
\\
\epsilon \nabla \veps &\rightarrow \nabla_y v_0 &\mbox{ in the two-scale sense},
\\
\te \veps &\rightarrow v_0 &\mbox{ strongly in } L^p(\Sigma,L^2((0,T),H^{\beta}(Z^{\ast})).
\end{align*}
Especially, we have $\veps \rightarrow v_0$ strongly in the two-scale sense in $L^p$ and $\veps|_{\neps} \rightarrow v_0|_{N}$ strongly in the two-scale sense on $\neps$ in $L^p$. 
\end{theorem}
\begin{proof}
The weak two-scale convergences of $\veps $ and $\epsilon \nabla \veps$ follow directly from Lemma \ref{LemmaTSWeakCompactness} and the estimates in \ref{TheoremStrongTSConvergenceAprioriEstimates}. The strong two-scale convergence of $\veps$ and $\veps|_{\neps}$ in $L^p$ follow from the strong convergence of $\te \veps$ in $L^p(\Sigma,L^2((0,T),H^{\beta}(Z^{\ast}))$ by the embedding $H^{\beta}(Z^{\ast}) \hookrightarrow L^2(N)$, and Lemma \ref{LemmaAequivalenzUnfoldingTSConvergence}. So it remains to prove the strong convergence of $\te \veps$. Therefore, we use the Kolmogorov-type compactness result \cite[Corollary 2.5]{GahnNeussRaduKolmogorovCompactness} for the sequence
\begin{align*}
\te \veps \in L^2(\Sigma,L^2((0,T),H^1(Z^{\ast})) \hookrightarrow L^p(\Sigma,L^2((0,T),H^{\beta}(Z^{\ast}))).
\end{align*}
We have to check the following three conditions:
\begin{enumerate}
[label = (K\arabic*)]
\item\label{KolmogorovConditionRange}
For every $A\subset \Sigma $ measurable, the sequence
\begin{align*}
\veps^A(t,y):= \int_A \te \veps d\x 
\end{align*}
is relatively compact in $L^2((0,T),H^{\beta}(Z^{\ast}))$.
\item\label{KolmogorovConditionShifts}
For $0<h\ll 1$ and $\bxi \in \R^{n-1}$ with $|\bxi|< h$ it holds that
\begin{align*}
\sup_{\epsilon} \|\te \veps (\cdot , \cdot + \bxi ,\cdot ) - \te \veps \|_{L^p(\Sigma_h,L^2((0,T),H^{\beta}(Z^{\ast})))} \overset{\bxi \to 0}{\longrightarrow} 0.
\end{align*}
\item\label{KolmogorovConditionBoundary}
For $0<h \ll 1$ it holds that
\begin{align*}
\sup_{\epsilon} \|\te \veps \|_{L^p(\Sigma\setminus \Sigma_h , L^2((0,T),H^{\beta}(Z^{\ast})))} \overset{h\to 0 }{\longrightarrow} 0.
\end{align*}
\end{enumerate}
Let us start with \ref{KolmogorovConditionBoundary}.  We obtain from the H\"older-inequality since $p<2$ that
\begin{align*}
\|\te \veps \|_{L^p(\Sigma\setminus \Sigma_h , L^2((0,T),H^{\beta}(Z^{\ast})))} \le C |h|^{\frac{2-p}{2p}} \|\te \veps \|_{L^2(\Sigma\setminus \Sigma_h , L^2((0,T),H^1(Z^{\ast})))} \le C|h|^{\frac{2-p}{2p}},
\end{align*}
where the last inequality follows from \ref{TheoremStrongTSConvergenceAprioriEstimates} and the properties of the unfolding operator from Lemma \ref{LemmaPropertiesUnfoldingOperator}. This gives \ref{KolmogorovConditionBoundary}. To prove \ref{KolmogorovConditionRange} we first notice that 
we have
\begin{align*}
\veps^A \in L^2((0,T),H^1(Z^{\ast})) \cap H^1((0,T),\Ho')
\end{align*}
with 
\begin{align}
\label{ZeitableitungVepsA}
\langle \partial_t \veps^A ,\phi \rangle_{\Ho',\Ho} = \langle \partial_t \te \veps ,\chi_A(\cdot_{\x}) \phi(\cdot_y) \rangle_{L^2(\Sigma,\Ho)',L^2(\Sigma,\Ho)}
\end{align}
for all $\phi \in  \Ho$. 
In fact, since $\chi_A (\cdot_{\x}) \phi(\cdot_y) \in L^2(\Sigma,\Ho)$ and $\partial_t \te \veps \in L^2((0,T),L^2(\Sigma,\Ho'))$ by Proposition \ref{PropositionExistenceTimeDerivativeUnfoldedSequence}, we obtain for every $\psi \in \mathcal{D}(0,T)$

\begin{align*}
\int_0^T (\veps^A , \phi )_{L^2(Z^{\ast})} \psi' dt & = \int_0^T \big(\te \veps , \chi_A (\cdot_{\x})  \phi(\cdot_y) \big)_{L^2(\Sigma \times Z^{\ast})} \psi' dt
\\
&= - \int_0^T \langle \partial_t \te \veps , \chi_A (\cdot_{\x} ) \phi(\cdot_y) \rangle_{L^2(\Sigma, \Ho'),L^2(\Sigma,\Ho)} \psi dt,
\end{align*}

what implies $\eqref{ZeitableitungVepsA}$.
 Obviously, due to \ref{TheoremStrongTSConvergenceAprioriEstimates} and Lemma \ref{LemmaPropertiesUnfoldingOperator}, the sequence $\veps ^A $ is bounded in $L^2((0,T),H^1(Z^{\ast}))$. Proposition \ref{PropositionExistenceTimeDerivativeUnfoldedSequence} and the estimate of $\partial_t \veps $ in \ref{TheoremStrongTSConvergenceAprioriEstimates} imply the boundedness of $\partial_t  \te \veps ^A$ in $L^2((0,T),L^2(\Sigma,\Ho)')$, from which we immediately obtain the boundedness of $\partial_t \veps^A $ in $L^2((0,T),\Ho')$ using $\eqref{ZeitableitungVepsA}$. Since $H^1(Z^{\ast}) \hookrightarrow H^{\beta}(Z^{\ast})$ is compact  for $\frac12 <\beta <1 $ and $H^{\beta}(Z^{\ast}) \hookrightarrow \Ho'$ is continuous, the Aubin-Lions lemma, see \cite{Lions}, implies \ref{KolmogorovConditionRange}. Now, we choose $0<h \ll 1$ and obtain for $|\bxi|<h$ with Lemma \ref{LemmaEstimateShiftsUnfolded}
\begin{align*}
\|\te \veps (\cdot,\cdot + \bxi ,\cdot )& - \te \veps \|_{L^2(\Sigma_{2h}, L^2((0,T),H^1(Z^{\ast})))} 
\\
&\le C \sum_{\bar{j}\in \{0,1\}^{n-1}} \left(\frac{1}{\sqrt{\epsilon}} \|\delta \veps \|_{L^2((0,T)\times \who_{\ast,\epsilon,h}^M)} + \sqrt{\epsilon}\|\nabla \delta \veps \|_{L^2((0,T)\times \who_{\ast,\epsilon,h}^M)}\right),
\end{align*}
for $l   = \bar{j} + \left[\frac{\bxi}{\epsilon}\right]$. Due to assumption (ii), the right-hand side converges to zero for $\epsilon,\bxi \to 0$. Next we show that this convergence implies in fact the uniform convergence in \ref{KolmogorovConditionShifts} with respect to $\epsilon$, see also \cite[p.710-711]{NeussJaeger_EffectiveTransmission} or \cite[p.1476-1477]{friesecke2002theorem}.  Let $0 < \rho $. Due to our previous results there exist $0 < \epsilon_0, \, \delta_0$, such that for all $\epsilon \le \epsilon_0 $ and $\vert\bxi \vert \le \delta$ it holds that
\begin{align}\label{UniformEstimateShifts}
\|\te \veps (\cdot , \cdot + \bxi ,\cdot ) - \te \veps \|_{L^p(\Sigma_h,L^2((0,T),H^{\beta}(Z^{\ast})))} \le \rho.
\end{align}
Since $\epsilon^{-1}  \in \N$, there are only finitely many elements $\epsilon_i$ with $i = 1,\ldots, N$, such that $\epsilon_0 < \epsilon_i$. For every $\epsilon_i$ there exists a $0 < \delta_i$, such that $\eqref{UniformEstimateShifts}$ is valid for $\epsilon = \epsilon_i$ and all $\vert \bxi\vert \le \delta_i$. Choosing $\delta:= \max_{i=0,\ldots,N} \{\delta_i\}$, inequality $\eqref{UniformEstimateShifts}$ holds uniformly with respect to $\epsilon$ for all $\vert \bxi \vert\le \delta$. This implies \ref{KolmogorovConditionShifts}.
The result follows from \cite[Corollary 2.5]{GahnNeussRaduKolmogorovCompactness}. 
\end{proof}

\section{Derivation of the macroscopic model}
\label{SectionDerivationMacroscopicModel}

The aim of this section is the derivation of the macroscopic model for $\epsilon \to 0$ with the methods developed in Section \ref{Sect_Two_scale_convergence}, which are based on the a priori estimates for the microscopic solutions established in Section \ref{SectionExistenceAprioriEstimates}. 
First of all we give a convergence result for  the sequences in the bulk-domains:
\begin{proposition}\label{PropositionConvergenceBulk}
Let $\ueps$ be the sequence of solutions of the microscopic problem $\eqref{MicroscopicModel}$. Then there exists $u_0^{\pm} \in L^2((0,T),H^1(\Omega^{\pm}))$ such that up to a subsequence
\begin{align*}
\chi_{\oe^{\pm}} \ueps^{\pm} &\rightarrow u_0^{\pm} &\mbox{ strongly in }& L^2((0,T)\times \Omega^{\pm}),
\\
\ueps^{\pm}(\cdot_{\x},\pm \epsilon) &\rightarrow u_0^{\pm}|_{\Sigma} &\mbox{ strongly in }& L^2((0,T)\times \Sigma),
\\
\chi_{\oe^{\pm}} \nabla \ueps^{\pm} &\rightharpoonup \nabla u_0^{\pm} &\mbox{ weakly in }& L^2((0,T)\times \Omega^{\pm}).
\end{align*}
\end{proposition}
\begin{proof}
This result was shown in
\cite[Proposition 2.1 and 2.2]{NeussJaeger_EffectiveTransmission} for time-derivatives $\partial_t \ueps^{\pm}$ in $L^2((0,T)\times \oe^{\pm})$. In our case, we have that $\partial_t \ueps^{\pm}$ are functionals on the space
\begin{align*}
\left\{ \peps^{\pm} \in H^1(\oe^{\pm}) \, : \, \peps^{\pm} = 0 \mbox{ on } \Sigma \times \{\pm \epsilon\}\right\}. 
\end{align*}
However, the methods from \cite{NeussJaeger_EffectiveTransmission} can easily be extended to our setting and we skip the details.
\end{proof}
In the next theorem we give the convergence results for the sequences in the channels. 
\begin{theorem}
\label{MainTheoremConvergenceChannels}
Let $\ueps$ be the sequence of solutions of the microscopic problem $\eqref{MicroscopicModel}$. Then there exists $u_0^M \in L^2((0,T)\times \Sigma, H^1(Z^{\ast}))$ such that up to a subsequence it holds for $p\in [1,2)$
\begin{align*}
\uepsm &\rightarrow u_0^M &\mbox{ strongly in the two-scale sense in }& L^p,
\\
\epsilon \nabla \uepsm &\rightarrow \nabla_y u_0^M &\mbox{ in the two-scale sense}&,
\\
\uepsm|_{\neps} &\rightarrow u_0^M|_N &\mbox{ strongly in the two-scale sense in }& L^p.
\end{align*}
Additionally, $\uepsm$ and $\uepsm|_{\neps}$ also converges weakly in the two-scale sense in $L^2$. Further, $u_0^M$ fulfills the following boundary condition on the top and the bottom of $Z^{\ast}$:
\begin{align}
\label{ContinuityMacroscopicSolution}
u_0^M(t,\x,y) = u_0^{\pm}(t,\x,0) \quad \mbox{ for almost every } (t,\x,y) \in (0,T)\times \Sigma \times S_{\ast}^{\pm},
\end{align}
\ie $u_0^M$ is constant on $S_{\ast}^{\pm}$.
\end{theorem}
\begin{proof}
For the convergence results we only have to check the conditions in Theorem \ref{TheoremStrongTSConvergence}. Condition \ref{TheoremStrongTSConvergenceAprioriEstimates} is just Lemma \ref{LemmaAprioriEstimates}. For \ref{TheoremStrongTSConvergenceControlShifts} we use inequality $\eqref{LemmaErrorEstimatesShiftsInequality}$ from Lemma \ref{LemmaErrorEstimatesShifts}.
To show that the initial terms in $\eqref{LemmaErrorEstimatesShiftsInequality}$ tend to $0$ for $\epsilon l \to 0$ we use 
\cite[Lemma 4.3]{Brezis},  and the Assumption \ref{AssumptionInitialCondition} on the initial conditions. We have 
\begin{align*}
\Vert \delta \ueps (0) \Vert_{\Lepsh}^2 = \foe \left\Vert \delta u_i^M\left(\bar{\cdot},\frac{\cdot_x}{\epsilon}\right)\right\Vert^2_{L^2(\who^M_{\ast,\epsilon,h})} + \sum_{\pm} \Vert \delta u_i^{\pm} \Vert^2_{L^2(\who_{\epsilon,h}^{\pm})}.
\end{align*}
The second term obviously tends to zero for $\epsilon l \to 0$. For the first term we have
\begin{align*}
\foe \left\Vert \delta u_i^M\left(\bar{\cdot},\frac{\cdot_x}{\epsilon}\right)\right\Vert^2_{L^2(\who^M_{\ast,\epsilon,h})}  &\le  \left\Vert \te \left(\delta  u_i^M\left(\bar{\cdot},\frac{\cdot_x}{\epsilon}\right) \right) \right\Vert_{L^2(\Sigma_h \times Z^{\ast})}
\\
&= \int_{\Sigma_h}\int_{Z^{\ast}} \left\vert u_i^M\left(\epsilon \left[\frac{\x}{\epsilon}\right] + \epsilon \bar{l} + \epsilon \bar{y} , y \right) - u_i^M(\x,y)\right\vert^2 d\x dy.
\end{align*}
Since $\epsilon \left[\frac{\x}{\epsilon}\right] + \epsilon \bar{l} + \epsilon \bar{y} - \x \overset{\epsilon l \to 0}{\longrightarrow} 0$ the right-hand side converges to zero for $\epsilon l \to 0$, 
see \cite[Lemma 4.3]{Brezis}.
%
The last term on the right-hand side in  $\eqref{LemmaErrorEstimatesShiftsInequality}$ goes to $0$, because of the strong convergence of $\ueps^{\pm}$ from Proposition \ref{PropositionConvergenceBulk} and the standard Kolmogorov-compactness result.

It remains to show the equation on $S^{\pm}_{\ast}$. Therefore we choose functions $\phi \in C^{\infty}((0,T)\times \Sigma \times \overline{Z^{\ast}})^n$ such that $\phi(t,x,\cdot)$ has compact support in $Z^{\ast} \cup S^+_{\ast} \cup S^-_{\ast}$, and extend this function by zero to $Z$ and then $Y$-periodically in $\bar{y}$-direction. From the two-scale results for $\uepsm$ and the strong convergence of $\ueps^{\pm}(\cdot_{\x},\pm \epsilon)$ we obtain by integration by parts
\begin{align*}
&\int_0^T \int_{\Sigma} \int_{Z^{\ast}} \nabla_y u_0^M \phi dy d\x dt  = \lim_{\epsilon \to 0} \foe \int_0^T \int_{\osem} \epsilon \nabla \uepsm \phi\left(t,\x,\fxe\right) dx dt
\\
=& \lim_{\epsilon \to 0} \bigg\{-\foe \int_0^T \int_{\osem} \uepsm \nabla_y \cdot \phi\left(t,\x,\fxe\right) dx dt + \sum_{\pm} \int_0^T \int_{\ssepm} \ueps^{\pm} \phi\left(t,\x,\fxe\right) \cdot \nu d\sigma dt \bigg\}
\\
=& -\int_0^T \int_{\Sigma} \int_{Z^{\ast}} u_0^M \nabla_y \cdot \phi dy d\x dt + \sum_{\pm}\int_0^T \int_{\Sigma} \int_{S_{\ast}^{\pm}} u_0^{\pm} \phi \cdot \nu d\sigma_y d\x dt.
\end{align*} 
Using again the integration by parts formula gives us the desired result.
\end{proof}

Next, we will state the macroscopic model, which is solved by the limit function $u_0:= (u_0^+,u_0^M,u_0^-)$ from Proposition \ref{PropositionConvergenceBulk} and Theorem \ref{MainTheoremConvergenceChannels}. The appropriate solution space is $L^2((0,T),\mathcal{H})$, where $\mathcal{H}$ is given by
\begin{eqnarray*}
\mathcal{H}&:=& \left\{u = (u^+,u^M,u^-) \in H^1(\Omega^+) \times L^2(\Sigma,H^1(Z^{\ast})) \times H^1(\Omega^-) \, : \right.\\
&& \hspace{7.8cm}\left. u^{\pm}|_{\Sigma} = u^M|_{S_{\ast}^{\pm}} \mbox{ on } \Sigma \times S^{\pm}_{\ast}\right\},
\end{eqnarray*}
with the inner product
\begin{align}\label{Inner_Product_H}
(u,\phi)_{\mathcal{H}} = \sum_{\pm} (u^\pm, \phi^\pm)_{H^1(\Omega^\pm)} + (u^M,\phi^M)_{\Sigma \times Z^{\ast}} + (\nabla_y u^M , \nabla_y \phi^M )_{\Sigma \times Z^{\ast}}.
\end{align}
For the derivation of the macroscopic problem in the limit $\epsilon \to 0$,
we have to exploit the convergence properties of the microscopic solutions $u_\epsilon$. Especially this requires to choose in the variational equation  \eqref{VariationalEquation} test-functions adapted to the structure of the transmission problem and whose restrictions to the channel domain is admissible for the definition of two-scale convergence. In fact, in general, for $u\in \mathcal{H}$ the function $u^M(\x,\fxe)$ is not well-defined in $\osem$ and $\neps$. Therefore, we consider the subspace $\mathcal{H}^{\infty}$ of $\mathcal{H}$ of smooth functions
\begin{align*}
\mathcal{H}^{\infty}:= C^{\infty}\big(\overline{\Omega^+}\big) \times C_0^{\infty}\big(\Sigma,C^{\infty}\big(\overline{Z^{\ast}}\big)\big) \times C^{\infty}\big(\overline{\Omega^-}\big) \cap \mathcal{H}.
\end{align*}
The following density result holds.
%

\begin{proposition}\label{Proposition_Density}
The space $\mathcal{H}^{\infty}$ is dense in $\mathcal{H}$ with respect to the norm induced by \eqref{Inner_Product_H}.
\end{proposition}
\begin{proof}
First of all, we have the orthogonal decomposition  
\begin{align}\label{OrthogonalDecomposition}
 \mathcal{H} = \overline{\mathcal{H}^{\infty}} \oplus \overline{\mathcal{H}^{\infty}}^{\perp}.
\end{align}
We will show that $\overline{\mathcal{H}^{\infty}}^{\perp} = \{0\}$ and therefore $ \mathcal{H}  = \mathcal{H}^{\infty} $. 
Due to $\eqref{OrthogonalDecomposition}$, for $u =(u^+,u^M,u^-)\in \overline{\mathcal{H}^{\infty}}^{\perp}$ we have for all $\phi = (\phi^+ , \phi^M, \phi^-) \in \overline{\mathcal{H}^{\infty}}$
\begin{align}\label{OrthogonalDecompositionEquation}
(u^M, \phi^M)_{L^2(\Sigma,H^1(Z^{\ast}))} + \sum_{\pm} (u^{\pm},\phi^{\pm})_{H^1(\Omega^{\pm})}  = 0 .
\end{align}
Especially, it holds for all $ \phi^M \in C_0^{\infty}(\Sigma , C^{\infty}( \overline{ Z^{\ast}}))  $ with $\phi^M = 0 $ on $S_{\ast}^{\pm}$ that (choose $\phi = (0,\phi^M,0)$ in $\eqref{OrthogonalDecompositionEquation}$)
\begin{align*}
(u^M,\phi^M)_{\Sigma \times Z^{\ast}} + (\nabla_y u^M , \nabla_y \phi^M )_{\Sigma \times Z^{\ast}} = 0 .
\end{align*}
This implies $\nabla_y u^M \in L^2(\Sigma,H(\mathrm{div},Z^{\ast}))$ and $\Delta_y u^M = u^M$. Hence, we have $\nabla_y u^M \cdot \nu \in L^2(\Sigma, H^{-\frac12}(\partial Z^{\ast}))$ and the divergence theorem implies that for all $\phi^M \in C_0^{\infty}(\Sigma , C^{\infty}( \overline{ Z^{\ast}}))$ with $\phi^M = 0 $ on $S_{\ast}^{\pm}$ it holds 
\begin{align}
\label{DensityResultNormalTracesZero}
\langle \nabla_y u^M \cdot \nu , \phi^M \rangle_{L^2(\Sigma,H^{-\frac12}(\partial Z^{\ast})),L^2(\Sigma,H^{\frac12}(\partial Z^{\ast}))} = 0.
\end{align}
By density, see \cite[Theorem 3.1]{Bernard2011}, equation $\eqref{DensityResultNormalTracesZero}$ is also true for all $\phi^M \in L^2(\Sigma,H^1(Z^{\ast}))$ with $\phi^M = 0$ on $S^{\pm}_{\ast}$.
Using again $\eqref{OrthogonalDecompositionEquation}$ and the divergence theorem, we obtain for arbitrary $\phi = (\phi^+,\phi^M,\phi^-) \in \mathcal{H}^{\infty}$ 
\begin{align}
\label{AuxiliaryEquationDensity}
\langle \nabla_y u^M \cdot \nu , \phi^M \rangle_{L^2(\Sigma,H^{-\frac12}(\partial Z^{\ast})),L^2(\Sigma,H^{\frac12}(\partial Z^{\ast}))} = - \sum_{\pm} (u^{\pm},\phi^{\pm})_{H^1(\Omega^{\pm})}.
\end{align}
Now, let us define the function 
\begin{align*}
\tilde{\phi}^M(\x,y) :=  \frac{y_n + 1}{2} \phi^+(\x,0) - \frac{y_n - 1}{2} \phi^-(\x,0).
\end{align*}
It holds that $(\phi^+,\tilde{\phi}^M,\phi^-) \in \mathcal{H}^{\infty}$ and $\phi^M - \tilde{\phi}^M = 0$ on $S_{\ast}^{\pm}$. Hence, due to $\eqref{DensityResultNormalTracesZero}$ and $\eqref{AuxiliaryEquationDensity}$ we have 
\begin{align*}
- \sum_{\pm} &( u^{\pm},\phi^{\pm})_{H^1(\Omega^{\pm})} 
\\
&= \left\langle \nabla_y u^M \cdot \nu , \frac{y_n + 1}{2} \phi^+(\x,0) - \frac{y_n - 1}{2} \phi^-(\x,0) \right\rangle_{L^2(\Sigma,H^{-\frac12}(\partial Z^{\ast})),L^2(\Sigma,H^{\frac12}(\partial Z^{\ast}))}.
\end{align*}
By density of $C^{\infty}(\overline{\Omega^{\pm}})$ in $H^1(\Omega^{\pm})$, the equation above holds for all $\phi^{\pm} \in H^1(\Omega^{\pm})$. Using again $\eqref{DensityResultNormalTracesZero}$, we get $\eqref{AuxiliaryEquationDensity}$ for all $\phi \in \mathcal{H}$ and especially for $\phi = u$. Now, the divergence theorem applied to the left-hand side of $\eqref{AuxiliaryEquationDensity}$ with $\phi^M = u^M$ and using $\Delta u^M = u^M$  implies $u=0$. 
%
%
%
%
%
%
\end{proof}

\begin{theorem}
\label{MainResultMacroscopicModel}
The limit function $u_0 = (u_0^+,u_0^M,u_0^-)$ from Proposition \ref{PropositionConvergenceBulk} and Theorem \ref{MainTheoremConvergenceChannels} is the unique weak solution of the following problem:
\begin{align*}
u_0 \in L^2((0,T),\mathcal{H}) \cap H^1((0,T),\mathcal{H}')
\end{align*}
and
\begin{subequations}
\begin{align}
\partial_t u_0^{\pm} - D^{\pm} \Delta u_0^{\pm} &= f^{\pm} (u_0^{\pm}) &\mbox{ in }& (0,T)\times \Omega^{\pm}, \label{EqMacropm}
\\ 
-D^{\pm} \nabla u_0^{\pm} \cdot \nu^{\pm} &= 0 &\mbox{ on }& (0,T)\times \partial \Omega^{\pm} \setminus \Sigma,
\\
u_0^{\pm}(0) &= u_i^{\pm} &\mbox{ in }& \Omega^{\pm},
\end{align}
with the interface conditions
\begin{align}
u_0^{\pm}|_{\Sigma} &= u_0^M|_{S_{\ast}^{\pm}} &\mbox{ on }& (0,T)\times \Sigma \times S_{\ast}^{\pm}, \label{TCContMacro}
\\
D^{\pm} \nabla u_0^{\pm} \cdot \nu^{\pm} &= \int_{S_{\ast}^{\pm}} D^M \nabla_y u_0^M \cdot \nu^M d\sigma &\mbox{ on }&  (0,T)\times \Sigma, \label{TCMacro}
\end{align}
where $\nu^{\pm}$ is the outer unit normal on $\partial \Omega^{\pm}$, and $\nu^M $ is the outer-unit normal on $\partial Z^{\ast}$, and $u_0^M$ solves the local cell problem
\begin{align}
\partial_t u_0^M -\nabla_y \cdot \big(D^M \nabla_y u_0^M\big) &= g(u_0^M) &\mbox{ in }& (0,T)\times \Sigma \times Z^{\ast}, \label{EqMacroM}
\\  
-D^M\nabla_y u_0^M \cdot \nu &= h(u_0^M) &\mbox{ on }& (0,T)\times \Sigma \times N, \label{BCNeumannM}
\\
u_0^M &= u_i^M &\mbox{ in }& \Sigma \times Z^{\ast}. 
\end{align}
\end{subequations}
In other words, $u_0$ is a solution of the following problem: For all $\phi = (\phi^+,\phi^M,\phi^-) \in \mathcal{H}$ it holds almost everywhere in $(0,T)$
\begin{align}\begin{aligned}\label{VarMacroModel}
\langle \partial_t u_0 , &\phi \rangle_{\mathcal{H}',\mathcal{H}}
+ \sum_{\pm}  (D^{\pm}\nabla u_0^{\pm},\nabla \phi^{\pm})_{ \Omega^{\pm}}
 + (D^M\nabla_y u_0^M, \nabla_y \phi^M)_{ \Sigma \times Z^{\ast}}
 \\
 =& \sum_{\pm} (f^{\pm}(u_0^{\pm}),\phi^{\pm})_{ \Omega^{\pm}} + (g(u_0^M),\phi^M)_{ \Sigma \times Z^{\ast}} - (h(u_0^M),\phi^M)_{ \Sigma \times N} , 
\end{aligned}\end{align}
and the initial condition $u_0(0) = (u_i^+,u_i^M,u_i^-)$ (which is well-defined since it holds  $u_0 \in C^0([0,T],L^2(\Omega^+) \times L^2(\Sigma \times Z^{\ast}) \times L^2(\Omega^-))$). 
Here, the weak equation $\eqref{VarMacroModel}$ is formally obtained in the usual way by testing the strong formulation with $\phi \in \mathcal{H}$, integrating with respect to $\Omega^{\pm}$ resp. $\Sigma \times Z^{\ast}$, and using integration by parts.
\end{theorem}
\begin{proof}
From Proposition \ref{PropositionConvergenceBulk} and Theorem \ref{MainTheoremConvergenceChannels} immediately follows that $u_0 \in L^2((0,T),\mathcal{H})$. Let $\phi \in C_0^{\infty}([0,T),\mathcal{H}^{\infty})$ and choose as a test-function in $\eqref{VariationalEquation}$
\begin{align*}
\peps(t,x):= \begin{cases}
\phi^+(t,\x,x_n - \epsilon) &\mbox{ for } x \in \oe^+,
\\
\phi^M\left(t,\x,\fxe\right) &\mbox{ for } x \in \osem,
\\
\phi^-(t,\x,x_n + \epsilon) &\mbox{ for } x \in \oe^-.
\end{cases}
\end{align*}
Integration with respect to time, integration by parts in time, and Proposition \ref{PropositionConvergenceBulk} and Theorem \ref{MainTheoremConvergenceChannels} together with Lemma \ref{LemmaTSStrongConvergence} imply for $\epsilon \to 0$
\begin{align*}
\sum_{\pm} & \left\{- ( u_0^{\pm},\partial_t \phi^{\pm} )_{(0,T)\times \Omega^{\pm}} + (D^{\pm}\nabla u_0^{\pm},\nabla \phi^{\pm})_{(0,T)\times \Omega^{\pm}}\right\}
\\  
 &- (u_0^M,\partial_t \phi^M)_{(0,T)\times \Sigma \times Z^{\ast}} + (D^M\nabla_y u_0^M, \nabla_y \phi^M)_{(0,T)\times \Sigma \times Z^{\ast}}
 \\
 =& \sum_{\pm} \left\{ (f^{\pm}(u_0^{\pm}),\phi^{\pm})_{(0,T)\times \Omega^{\pm}} + (u_i^{\pm},\phi^{\pm}(0))_{\Omega^{\pm}} \right\}
 \\
&+ (g(u_0^M),\phi^M)_{(0,T)\times \Sigma \times Z^{\ast}} - (h(u_0^M),\phi^M)_{(0,T)\times \Sigma \times N} + (u_i^M,\phi^M(0))_{ \Sigma \times Z^{\ast}}. 
\end{align*}
By density, see Proposition \ref{Proposition_Density}, this equation also holds for all $\phi \in C_0^{\infty}([0,T), \mathcal{H})$. 
Especially, this equation implies $\partial_t u_0 \in L^2((0,T),\mathcal{H}')$ and $u_0(0) = (u_i^+,u_i^M,u_i^-)$. 
 Hence, $u_0$ is a weak solution of the macroscopic model in the theorem.  Uniqueness follows by standard arguments.
\end{proof}

We emphasize that a regular solution of \eqref{VarMacroModel} satisfies the strong formulation in the theorem. This is obtained in the following way: We first choose as test-functions $\phi$ in \eqref{VarMacroModel} functions from $C_0^{\infty}([0,T), \mathcal{H})$ with compact support in $\Omega^\pm$ respectively in $Z^\ast$. This immediately yields the equations \eqref{EqMacropm} respectively \eqref{EqMacroM}. Afterwards, we test the weak formulation \eqref{VarMacroModel} with functions of the form $(0, \phi^M, 0)\in C_0^{\infty}([0,T), \mathcal{H})$ with $\textrm{supp}(\phi^M(\bar x, \cdot)) \cap (S_\ast^\pm) = \emptyset $ to exhibit the nonhomogeneous Neumann-boudary condition \eqref{BCNeumannM} on $\Sigma \times N$. Finally, to derive the transmission conditions \eqref{TCMacro}, we test the weak formulation with functions of the form $ (\phi^+, \phi^M, 0)\in C_0^{\infty}([0,T), \mathcal{H})$, with $\phi^+ \not\equiv 0$ on $\Sigma$, to obtain the transmission condition 
\begin{align*}
D^{+} \nabla u_0^{+} \cdot \nu^{+} = \int_{S_{\ast}^{+}} D^M \nabla_y u_0^M \cdot \nu^M d\sigma \mbox{ on } (0,T) \times \Sigma
\end{align*}
and analogously with functions of the form $ (0, \phi^M, \phi^-)\in C_0^{\infty}([0,T), \mathcal{H})$, with $\phi^- \not\equiv 0$ on $\Sigma$, to obtain the second transmission condition in \eqref{TCMacro}. We emphasize that the interface conditions \eqref{TCContMacro} have been derived in Theorem \ref{MainTheoremConvergenceChannels}.

\begin{remark}\mbox{}
\begin{enumerate}
[label = (\roman*)]
\item By 
choosing test-functions in $\eqref{VarMacroModel}$ with $\phi^+ = \phi^-$ on $\Sigma$ (and $\phi^M$ for example constant with respect to $y$, \ie $\phi^M = \phi^+ = \phi^-$ in $Z^{\ast}$) 
we obtain in a weak sense the following relation for the jump of the normal fluxes:
\begin{align*}
\big(D^+ \nabla u_0^+ - D^- \nabla u^-\big)\cdot \nu^+ = \sum_{\pm}\int_{S^{\pm}_{\ast} } D^M \nabla_y u_0^M \cdot \nu^M d\sigma  \quad \mbox{ on } (0,T)\times \Sigma.
\end{align*} 
Formally this is obtained by subtracting the two equations in \eqref{TCMacro}.
\item We only considered a single concentration $\ueps:(0,T)\times \oe \rightarrow \R$. However, the result can easily be generalized to systems and vector-valued functions $\tilde{u}_{\epsilon}: (0,T)\times \oe \rightarrow \R^m$. In this case, the nonlinear reaction-kinetics have the form $F:[0,T] \times \oe \times \R^m \rightarrow \R^m$ and have to be uniformly Lipschitz-continuous with respect to the last variable. The main difference lies in the derivation of the a priori estimates. However, this can be done by adding up the single equations, see for example \cite{GahnNeussRaduKnabner2018a} for more details.
\item For the sake of simplicity we just considered constant scalar diffusion $D^{\pm}$ in the bulk-domains. It is obvious that we can also consider a diffusion tensor $D^{\pm} \in \R^{n\times n}$ positive and symmetric. However, we can also consider oscillating diffusion coefficients $D^{\pm}\left(\fxe\right)$ with $D^{\pm} \in L^{\infty }((0,1)^n)^{n\times n }$. In this case one can use standard homogenization theory in the bulk-domains to obtain a reaction-diffusion problem in the macroscopic bulk-domains $\Omega^{\pm}$ with effective diffusion coefficients.
\item We emphasize that our method is not restricted to the specific form of the microscopic problem $\eqref{MicroscopicModel}$, but the methods developed in  Section \ref{Sect_Two_scale_convergence} can also be applied to other problems for which the  a priori estimates from 
 \ref{LemmaAprioriEstimates} and \ref{LemmaErrorEstimatesShifts} 
are  satisfied.  
\end{enumerate}
\end{remark}

\section*{Acknowledgements}
The first author was supported by the SCIDATOS project, funded by the Klaus Tschira Foundation (Grant number 00.0277.2015).

\bibliographystyle{abbrv}
\bibliography{literature} 

\end{document}